\documentclass[12pt]{amsart}

\usepackage{amsthm}

\usepackage{amssymb}
\usepackage{verbatim}
\usepackage[curve,matrix,arrow]{xy}
\usepackage{amsmath,amsfonts}
\usepackage{graphicx}
\usepackage{epsf}

\newcommand{\thmref}[1]{Theorem~\ref{#1}}
\newcommand{\propref}[1]{Proposition~\ref{#1}}
\newcommand{\lemref}[1]{Lemma~\ref{#1}}
\newcommand{\eqnref}[1]{Equation~(\ref{#1})}
\newcommand{\remref}[1]{Remark~\ref{#1}}
\newcommand{\corref}[1]{Corollary~\ref{#1}}

\newcommand{\figref}[1]{Figure~\ref{#1}}

\newcommand{\conjref}[1]{Conjecture~\ref{#1}}

  {\end{list}}

\def\li{L_{i}}
\def\ri{R_{i}}
\def\ddx{{\frac{d}{dx}}}
\def\oga{{\overline{\ga}}}
\def\tga{{\widetilde{\ga}}}

\def\RR{{\mathbb R}}

\newtheorem{theorem}{Theorem}[section]

\newtheorem{corollary}[theorem]{Corollary}
\newtheorem{conjecture}[theorem]{Conjecture}
\newtheorem{lemma}[theorem]{Lemma}
\newtheorem{proposition}[theorem]{Proposition}

\newtheorem{remark}[theorem]{Remark}

\theoremstyle{definition}

\theoremstyle{notation}
\newtheorem*{notation}{Notation}

\newcommand{\jj}[3]{j_{#1}(#2,#3)}
\newcommand{\hj}[3]{\hat{j}_{#1}(#2,#3)}

\newcommand{\ga}{\Gamma}

\newcommand{\tg}{\tau(\Gamma)}
\newcommand{\ta}[1]{\tau(#1)}

\newcommand{\ee}[1]{E(#1)}
\newcommand{\vv}[1]{V(#1)}
\newcommand{\va}{\upsilon}

\newcommand{\pp}{p_{i}}

\newcommand{\qq}{q_{i}}

\def\can{{\mathop{\rm can}}}

\def\<{\langle }
\def\>{\rangle }
\newcommand{\secref}[1]{\S\ref{#1}}

\def\elg{\ell (\ga)}

\begin{document}

\title[The tau constant and the edge connectivity]
{The Tau Constant And The Edge Connectivity of A Metrized Graph}

\author{Zubeyir Cinkir}
\address{Zubeyir Cinkir\\
Department of Mathematics\\
University of Georgia\\
Athens, Georgia 30602\\
USA}
\email{cinkir@math.uga.edu}


\thanks{I would like to thank Dr. Robert Rumely for his continued support and help.}

\begin{abstract}
The tau constant is an important invariant of
a metrized graph, and it has applications in arithmetic properties of curves.
We show how the tau constant of a metrized graph changes under
successive edge contractions and deletions. We discover identities which we call  ``contraction", ``deletion", and
``contraction-deletion" identities on a metrized
graph. By establishing a lower bound for the tau constant in terms of the edge
connectivity, we prove that Baker and Rumely's lower bound conjecture on the tau constant holds for metrized graphs
with edge connectivity $5$ or more. We show that proving this conjecture for $3$-regular graphs is enough to prove it for all graphs.
\end{abstract}

\maketitle

\section{Introduction}\label{sec introduction}

Metrized graphs, which are graphs equipped with a distance function
on their edges, appear in many places in arithmetic geometry.
Metrized graphs and their invariants are studied in the articles
\cite{Zh1}, \cite{Zh2}, \cite{C1}, \cite{C2}, \cite{C3}, \cite{C5}, and \cite{C6}.
Metrized graphs which arise as dual graphs of special fibers of curves, and Arakelov
Green's functions $g_{\mu}(x,y)$ on metrized graphs, play an
important role in both articles \cite{CR} and \cite{Zh1} to study arithmetic
properties of curves. T. Chinburg and R. Rumely \cite{CR}
introduced a canonical measure $\mu_{can}$ of total mass $1$ on a
metrized graph $\ga$.
The diagonal values $g_{\mu_{can}}(x,x)$ are constant on
$\ga$. M. Baker and Rumely called this constant ``the tau constant''
of a metrized graph $\ga$, and denoted it by $\tg$. They posed the following conjecture
concerning lower bound of $\tg$.
\begin{conjecture}\cite{BRh} \label{TauBound}
There is a universal constant $C>0$
such that for all metrized graphs $\Gamma$,
$\tau(\Gamma) \ \geq \ C \cdot \ell(\Gamma)\,$
where $\ell(\Gamma)$ is the total length of $\ga$.
\end{conjecture}
We call \conjref{TauBound} Baker and Rumely's
lower bound conjecture.

In summer 2003 at UGA,  an REU group lead by Baker and Rumely
studied properties of the tau constant and the lower bound conjecture. Baker and Rumely
\cite{BRh} introduced a measure valued Laplacian operator $\Delta$ which
extends Laplacian operators studied earlier in the articles \cite{CR} and
\cite{Zh1}. This Laplacian operator combines the ``discrete''
Laplacian on a finite graph and the ``continuous'' Laplacian
$-f''(x)dx$ on $\RR$. Baker and Rumely \cite{BRh} studied
harmonic analysis on metrized graphs. In terms of spectral theory,
the tau constant is the trace of the inverse operator of $\Delta$
when $\ga$ has total length $1$.

In this paper, we prove that \conjref{TauBound} holds with $C=\frac{1}{108}$ for a graph $\ga$ with edge connectivity
at least $6$, and with $C=\frac{1}{300}$ for a graph $\ga$ with edge connectivity $5$. The proof involves establishing
a set of identities, which we call  ``contraction'', ``deletion'', and ``contraction-deletion'' identities on a metrized graph.
By using these identities, we show how the tau constant changes after successive edge deletions and contractions. In particular,
when we consider successive edge contractions until we are left with only two vertices, we use our previous results \cite{C2} about the tau constant to obtain a set of inequalities between the terms adding up to the tau constant. In this way, we transform the tau lower bound problem into a linear optimization problem. Finally, we obtain a lower bound to the tau constant
$\tg$ of a metrized graph $\ga$ in terms of the edge connectivity and $\ell(\Gamma)$ when the edge connectivity is at least $5$.
The results here not only extend those obtained in \cite[Sections 3.6, 3.7, 3.9, 3.10 and 3.12]{C1}
but we obtained in a more coherent and systematic manner.

Applications of these results to arithmetic of curves and specifically to Bogomolov Conjecture can be found in the article \cite{C5}.

Note that there is a $1-1$ correspondence between equivalence classes of finite connected weighted graphs, metrized graphs, and resistive electric circuits.
If an edge $e_i$ of a metrized graph has length $\li$, then the resistance of $e_i$ is $\li$
in the corresponding resistive electric circuit, and the weight of $e_i$ is $\frac{1}{\li}$ in the corresponding
weighted graph. Therefore, the identities we show in this paper have equivalent forms for weighted graphs.

\section{Metrized graphs and their tau constants}\label{sec metrizedgraph}

In this section, we recall a few facts about metrized graphs,
the canonical measure on a metrized graph $\ga$,
and the tau constant of $\ga$.

A metrized graph $\ga$ is a finite connected graph
equipped with a distinguished parametrization of each of its edges.
A metrized graph $\ga$ can have multiple edges and self-loops. For any given $p \in \ga$,
the number of directions emanating from $p$ will be called the \textit{valence} of $p$, and will be denoted by $\va(p)$. By definition, there can be only finitely many $p \in \ga$ with $\va(p)\not=2$.

For a metrized graph $\ga$, we will denote a vertex set for $\ga$ by $\vv{\ga}$.
We require that $\vv{\ga}$ be finite and non-empty and that $p \in \vv{\ga}$ for each $p \in \ga$ if $\va(p)\not=2$. For a given metrized graph $\ga$, it is possible to enlarge the
vertex set $\vv{\ga}$ by considering additional valence $2$ points as vertices.

For a given metrized graph $\ga$ with vertex set $\vv{\ga}$, the set of edges of $\ga$ is the set of closed line segments with end points in $\vv{\ga}$. We will denote the set of edges of $\ga$ by $\ee{\ga}$. However, if
$e_i$ is an edge, by $\ga-e_i$ we mean the graph obtained by deleting the {\em interior} of $e_i$.

Let $v:=\# (\vv{\ga})$ and $e:=\# (\ee{\ga})$. We define the genus $\ga$ to be the first Betti number $g:=e-v+1$ of the graph $\ga$. Note that the genus is a topological invariant of $\ga$. In particular, it is independent of the choice of the vertex set $\vv{\ga}$.
Since $\ga$ is connected, $g(\ga)$ coincides with the cyclotomic number of $\ga$ in combinatorial graph theory. We will simply use $g$ to show $g(\ga)$ when there is no danger of confusion.

We denote the length of an edge $e_i \in \ee{\ga}$ by $\li$. The total length of $\ga$, which will be denoted by $\elg$, is given by $\elg=\sum_{i=1}^e\li$.

Let $\ga$ be a metrized graph. If we scale each edge
in $\ga$ by multiplying its length by $\frac{1}{\ell(\ga)}$, we
obtain a new graph which is called normalization of $\ga$ and
denoted by $\ga^{N}$. Note that $\ga$ and $\ga^N$ have the same topology, and $\ell(\ga^N)=1$.
If $\ga=\ga^N$, we call $\ga$ be a normalized graph.

A metrized graph $\ga$ is called $n$-regular if it has a vertex set $\vv{\ga}$ such that
$\va(p)=n$ for all vertices $p \in \vv{\ga}$.


We will denote the minimum of the valences of vertices in $\vv{\ga}$ by $\delta(\ga)$.
The minimum number of edges whose deletion
disconnects $\ga$ is called the ``edge connectivity" of $\ga$ and
denoted by $\Lambda(\ga)$.
The minimum number of vertices whose deletion
disconnects $\ga$ is called the ``vertex connectivity" of $\ga$ and
denoted by $\kappa(\ga)$.

In the article \cite{CR}, a kernel $j_{z}(x,y)$ giving a
fundamental solution of the Laplacian is defined and studied as a
function of $x, y, z \in \Gamma$. For fixed $z$ and $y$ it has the
following physical interpretation: When $\Gamma$ is viewed as a
resistive electric circuit with terminals at $z$ and $y$, with the
resistance in each edge given by its length, then $j_{z}(x,y)$ is
the voltage difference between $x$ and $z$, when unit current enters
at $y$ and exits at $z$ (with reference voltage 0 at $z$).

For any $x$, $y$, $z$ in $\ga$, the voltage function $j_x(y,z)$ on
$\ga$ is a symmetric function in $y$ and $z$, and it satisfies
$j_x(x,z)=0$ and $j_x(y,y)=r(x,y)$, where $r(x,y)$ is the resistance
function on $\ga$. For each vertex set $\vv{\ga}$, $j_{z}(x,y)$ is
continuous on $\ga$ as a function of $3$ variables.
As the physical interpretation suggests, $j_x(y,z) \geq 0$ for all $x$, $y$, $z$ in $\ga$.
For proofs of these facts, see the articles \cite{CR}, \cite[sec 1.5 and sec 6]{BRh}, and \cite[Appendix]{Zh1}.
The voltage function $j_{z}(x,y)$ and the resistance function $r(x,y)$ on a metrized graph
were also studied in the articles \cite{BF} and \cite{C2}.

For any real-valued, signed Borel measure $\mu$ on $\Gamma$ with
$\mu(\Gamma)=1$ and $|\mu|(\Gamma) < \infty$, define the function
$j_{\mu}(x,y) \ = \ \int_{\Gamma} j_{\zeta}(x,y) \, d\mu({\zeta}).$
Clearly $j_{\mu}(x,y)$ is symmetric, and is jointly continuous in
$x$ and $y$. Chinburg and Rumely \cite{CR} discovered that there is a unique measure $\mu=\mu_{can}$ with above properties
such that $j_{\mu}(x,x)$ is constant on $\ga$. The measure $\mu_\can$ is called the
\textit{canonical measure}.
Baker and Rumely \cite{BRh} called the constant $\frac{1}{2}j_{\mu}(x,x)$ the \textit{tau constant} of $\ga$
and denoted it by $\tg$.

The following lemma gives another description of the tau constant. In particular, it implies that the tau constant is positive.
\begin{lemma}\cite[Lemma 14.4]{BRh}\label{lemtauformula}
For any fixed $y$ in $\ga$,
$\tg =\frac{1}{4}\int_{\ga}\big(\frac{\partial}{\partial x} r(x,y) \big)^2dx$.
\end{lemma}

We will use the following results frequently in later sections.
\begin{lemma}\cite[pg. 37]{BRh} \cite[Corollaries 2.17 and 2.22]{C2}\label{lem tau for tree and circle}
If $\ga$ is a tree, i.e. a graph without cycles, then $\tg =\frac{\elg}{4}$. If $\ga$ is a circle graph, then
$\tg =\frac{\elg}{12}$.
\end{lemma}
\begin{remark}\label{remcutvertex1}
Whenever a graph $\ga$ has vertex $p$ such that removing $p$
disconnects $\ga$, i.e. $p$ is a cut-vertex of $\ga$, then
$\ga=\ga_1\cup \ga_2$ for subgraphs $\ga_1$ and $\ga_2$ with
$\ga_1\cap \ga_2=\{p\}$. In this case, we have $\ta{\ga_1 \cup
\ga_2}=\ta{\ga_1}+\ta{\ga_2}$, which we call the \textit{additive
property} of the tau constant (see also \cite[pg. 11]{C2}).
It was initially noted in \cite{REU}.
\end{remark}
Therefore, proving Conjecture
~\ref{TauBound} for graphs with vertex connectivity $\kappa(\ga)\geq 2$ yields it for all graphs.
\begin{remark}\cite{BRh}\label{rem tau scale-idependence}
If we multiply all lengths on $\Gamma$ by a
positive constant $c$, we obtain a graph $\Gamma'$ of
total length $c \cdot \ell(\Gamma)$. Then $\tau(\Gamma')
= c \cdot \tau(\Gamma)$. This will be called the \textit{scale-independence} of the tau constant.
By this property, to prove \conjref{TauBound}, it is enough to consider metrized graphs with total length $1$.
\end{remark}
\begin{remark}\label{remvalence}
Let $\ga$ be any metrized graph with resistance function $r(x,y)$.
The tau constant $\tg$ is independent of the vertex set $\vv{\ga}$ chosen. In particular, enlarging $\vv{\ga}$ by including points $p \in \ga$ with $\va(p)=2$ does not change $\tg$. Thus, $\tg$ depends only
on the topology and the edge length distribution of the metrized
graph $\ga$. This will be called the \textit{valence property} of the tau constant.
\end{remark}

We will denote by $R_{i}(\ga)$, or by $R_i$ if there is no danger of confusion, the resistance between the end points of an edge $e_i$ of a graph $\ga$  when the interior of the edge $e_i$ is deleted from $\ga$. We will use the following notation in the rest of this paper:
\begin{equation}\label{eqn defining z and r}
\begin{split}
z(\ga)= \sum_{e_i \in \ee{\ga}}\frac{\li^2}{\li+\ri}, \qquad r(\ga)= \sum_{e_i \in \ee{\ga}}\frac{\li \ri}{\li+\ri}.
\end{split}
\end{equation}
Note that $\ell(\ga)=z(\ga)+r(\ga).$

Chinburg and Rumely \cite[page 26]{CR} showed that
\begin{equation}\label{eqn genus}
\sum_{e_i \in \ee{\ga}}\frac{\li}{\li +\ri}=g, \quad \text{equivalently } \sum_{e_i \in \ee{\ga}}\frac{\ri}{\li +\ri}=v-1.
\end{equation}
\begin{notation}
Define $A_{p,q,\ga}:=\int_{\ga}\jj{x}{p}{q}(\ddx \jj{p}{x}{q})^2dx$ for any p, q $\in \ga$.
\end{notation}
Properties of $A_{p,q,\ga}$ were studied in the article \cite[Sections 4 and 8]{C2}.
For any p, q $\in \ga$, $0 \leq A_{p,q,\ga} \leq r(p,q) \big(r_{\ga}(p)-\frac{r(p,q)}{2}\big)$, where  $r_{\ga}(p)=\max \{ r(p,x)|x \in \ga \}$ and $r(x,y)$ is the resistance function in $\ga$. Here, the upper bound follows by combining \cite[Theorem 4.3 part (vi)]{C2} and \cite[Corollary 2.19]{C2}.

We call an edge $e_i \in \ee{\ga}$ a bridge if $\ga-e_i$ is disconnected. If $\ga-e_i$ is connected for every $e_i \in \ee{\ga}$, we call $\ga$ a bridgeless graph.
\begin{theorem}\cite[Theorem 5.7]{C2}\label{thmbasic2}
Let $\ga$ be a bridgeless graph. Suppose that $\pp$, $\qq$ are the
end points of the edge $e_{i}$, for each $i=1,2,\dots, e$. Then,
\begin{equation*}
\tg=\frac{\ell(\ga)}{12}-\sum_{i=1}^{e}\frac{\li
A_{\pp,\qq,\ga-e_{i}}}{(\li+\ri)^2} .
\end{equation*}
\end{theorem}

\begin{theorem}\cite[Theorem 2.21]{C2}\label{thmbasic}
For any $p,q \in \ga$,
$
\tg = \frac{1}{4}\int_{\ga}(\frac{d}{dx}\jj{x}{p}{q})^2dx +
\frac{1}{4}r(p,q).
$
\end{theorem}
\section{Edge contractions and deletions}\label{sec edge cont and del}

Let $\oga_i$ be the graph obtained by contracting the i-th edge
$e_{i}$, $i \in \{1,2, \dots e\}$, of a given graph $\ga$ to its end
points. If $e_{i} \in \ga$ has end points $\pp$ and $\qq$, then in
$\oga_i$, these points become identical, i.e., $\pp=\qq$. Let $\tga_i$ be the graph obtained by
identifying the end points of the $i-$th edge $e_{i} \in \ee{\ga}$.
This makes $e_i$ into a loop in $\tga_i$. Note that $\ta{\oga_i}=\ta{\tga_i}-\frac{\li}{12}$
by the additive property of the tau constant and \lemref{lem tau for tree and circle}.

The following lemma sheds light on how the tau constant changes by contraction of an edge:
\begin{lemma}\cite[Lemma 6.2]{C2}\label{lemcontract1}
Let $e_i \in \ee{\ga}$ be such that $\ga-e_i$ is connected. Then we have
$$\ta{\ga} =\ta{\oga_i} + \frac{\li}{12} -\frac{\li
A_{\pp,\qq,\ga-e_i}}{\ri(\li+\ri)}.$$
\end{lemma}
Note that \lemref{lemcontract1} involves terms containing $A_{\pp,\qq,\ga-e_i}$, which are fairly difficult to understand. One wants to understand the effect of edge contraction in a better way. An important step in this direction is provided by \propref{prop tau contract} and \thmref{thm tau contract}, which depend on \thmref{thmbasic2}.
\begin{proposition}\label{prop tau contract}
Let $\ga$ be a bridgeless graph with v$:=\#(\vv{\ga}) \geq 3$. Then,
$$\ta{\ga} =\frac{1}{v-2}\sum_{i=1}^{e}\frac{\ri}{\li+\ri}\ta{\oga_i}
- \frac{z(\ga)}{12(v-2)}, \quad \ta{\ga} =\frac{1}{v-2}\sum_{i=1}^{e}\frac{\ri}{\li+\ri}\ta{\tga_i} - \frac{\ell(\ga)}{12(v-2)}.$$
\end{proposition}
\begin{proof}
Multiply both sides of the equation in \lemref{lemcontract1} by
$\frac{\ri}{\li+\ri}$,
sum over all edges of $\ga$, and use the fact that
$\sum_{i=1}^{e}\frac{\ri}{\li+\ri} = v-1$ (see \eqnref{eqn genus}) to obtain
\begin{equation*}
\begin{split}
(v-1)\ta{\ga} =\sum_{i=1}^{e}\frac{\ri}{\li+\ri}\ta{\oga_i} +
\frac{r(\ga)}{12}-\sum_{i=1}^{e}\frac{\li A_{\pp,\qq,\ga-e_i}}{(\li+\ri)^2}.
\end{split}
\end{equation*}
Recall that $z(\ga)$ and $r(\ga)$ are defined in \eqnref{eqn defining z and r}.
We obtain the first formula by using \thmref{thmbasic2}. Then the second formula follows from the fact that $\ta{\oga_i}=\ta{\tga_i}-\frac{\li}{12}$.
\end{proof}
In the proof of \propref{prop tau contract}, we used
the fact that $\ga$ is bridgeless when we worked with terms
$A_{\pp,\qq,\ga-e_i}$. We will now extend the result of
\propref{prop tau contract} to any connected graph $\ga$. For an
edge $e_i$ which is a bridge in $\ga$, the end points $\pp$ and
$\qq$ become disconnected in $\ga-e_i$, and so $\ri=\infty$. In
such cases, if we use the limiting values of the corresponding
terms, it is possible to extend \propref{prop tau contract} to a
metrized graph with bridges. More precisely, note that
\begin{equation*}
\begin{split}
\ta{\ga} &=\frac{1}{v-2}\sum_{i=1}^{e}\big[\lim_{t\rightarrow
\ri}\frac{t}{\li+t}\big]\ta{\oga_i} -
\frac{1}{12(v-2)}\sum_{i=1}^{e}\big[\lim_{t\rightarrow
\ri}\frac{\li^2}{\li+t}\big] .
\end{split}
\end{equation*}
In short, we set $\frac{\ri}{\li+\ri}:=1$
and $\frac{\li}{\li+\ri}:=0$ whenever $\ri=\infty$.
\begin{theorem}\label{thm tau contract}
Let $\ga$ be a metrized graph with v$:=\#(\vv{\ga}) \geq 3$. Then we have
$$\ta{\ga} =\frac{1}{v-2}\sum_{i=1}^{e}\frac{\ri}{\li+\ri}\ta{\oga_i}
- \frac{z(\ga)}{12(v-2)},
\quad \ta{\ga} =\frac{1}{v-2}\sum_{i=1}^{e}\frac{\ri}{\li+\ri}\ta{\tga_i} - \frac{\ell(\ga)}{12(v-2)}.$$
\end{theorem}
\begin{proof}
We already dealt with the case in which $\ga$ is bridgeless. Suppose
that $\ga$ has bridges. Let $B=\{e_{i_{1}}, e_{i_{2}}, \dots,
e_{i_{k}} \}$ be the set of all bridges in $\ga$, for some positive
integer $k$. Let $\gamma$ be the graph obtained from $\ga$ by
contracting all of its bridges to their end points. Thus, an edge
$e_i$ belongs to $\ee{\gamma}$ iff $e_i \not\in B $. By the additive property of
$\tg$ (i.e., by \remref{remcutvertex1}) and \lemref{lem tau for tree and circle},
\begin{equation}\label{eqncontract1}
\begin{split}
\ta{\ga}=\ta{\gamma} +\frac{1}{4}\sum_{e_j \in B}L_j .
\end{split}
\end{equation}
Clearly, $\gamma$ is connected and bridgeless with $v-k$
vertices. Note that if $e_i \in B$, then
\begin{equation}\label{eqncontract2a}
\begin{split}
\ta{\oga_i} = \ta{\gamma} +\frac{1}{4}\sum_{e_j \in B}L_{j} -\frac{\li}{4},
\end{split}
\end{equation}
and if $e_i \not\in B$, then
\begin{equation}\label{eqncontract2b}
\begin{split}
\ta{\oga_i} = \ta{\overline{\gamma}_i} +\frac{1}{4}\sum_{e_j \in B}L_j .
\end{split}
\end{equation}
In either case,
\begin{equation}\label{eqncontract2c}
\begin{split}
z(\ga)=z(\gamma),
\text{ and }
 \sum_{e_i \in \ee{\ga}-B}\frac{\ri(\ga)}{\li+\ri(\ga)}=\sum_{e_i \in \ee{\gamma}}\frac{\ri(\gamma)}{\li+\ri(\gamma)}=v-k-1.
\end{split}
\end{equation}

Since $\gamma$ is bridgeless, we can apply \propref{prop tau contract} to obtain
\begin{equation}\label{eqncontract3}
\begin{split}
\ta{\gamma}=\frac{1}{v-k-2}\sum_{e_i \in
\ee{\gamma}}\frac{\ri}{\li+\ri}\ta{\overline{\gamma}_i}
-\frac{z(\gamma)}{12(v-k-2)}.
\end{split}
\end{equation}
Then by Equations (\ref{eqncontract2a}) and (\ref{eqncontract2b})
\begin{equation*}
\begin{split}
\sum_{e_i \in \ee{\ga}}\frac{\ri}{\li+\ri}\ta{\oga_i}
& =\sum_{i \not\in B}\frac{\ri}{\li+\ri}\big( \ta{\overline{\gamma}_i}
+\frac{1}{4}\sum_{j \in B}L_j\big)+\sum_{i \in
B}\frac{\ri}{\li+\ri}\big( \ta{\gamma} +\frac{1}{4}\sum_{j \in B}L_j
-\frac{L_{i}}{4}\big),
\\ &=\sum_{i \not\in B}\frac{\ri}{\li+\ri}\ta{\overline{\gamma}_i}
+\Big( \frac{1}{4}\sum_{j \in B}L_j \Big) \Big(\sum_{i \not\in
B}\frac{\ri}{\li+\ri}\Big) +k\ta{\gamma}+\frac{k-1}{4}\sum_{j \in
B}L_j
\\ &=(v-2)\ta{\gamma}+\frac{z(\gamma)}{12}
+\frac{v-2}{4}\sum_{j \in B}L_j, \quad \text{by Equations (\ref{eqncontract3}) and (\ref{eqncontract2c})}.
\\ &=(v-2)\tg + \frac{z(\ga)}{12}, \quad \text{by Equations (\ref{eqncontract1}) and  (\ref{eqncontract2c}).}
\end{split}
\end{equation*}
This is equivalent to the first formula we wanted to show.
By using the fact that $\ta{\tga_i}=\ta{\oga_i}+\frac{\li}{12}$, for all $e_i \in \ee{\ga}$, along with the first formula, we obtain the second formula.
\end{proof}
%
%
The following lemma shows how the tau constant changes by deletion of an edge when the remaining graph is connected.
\begin{lemma}\cite[Corollary 5.3]{C2}\label{lemcor2twopunion}
Suppose that $\ga$ is a graph such that $\ga-e_{i}$, for some edge
$e_{i} \in \ee{\ga}$ with length $\li$ and end points $p_{i}$ and
$q_{i}$, is connected. Then we have
\begin{equation*}
\begin{split}
\ta{\ga} = \ta{\ga-e_{i}} +\frac{\li}{12}-\frac{\ri}{6}+
\frac{A_{p_{i},\qq,\ga-e_{i}}}{\li+\ri}.
\end{split}
\end{equation*}
\end{lemma}
We can combine \lemref{lemcontract1} and \lemref{lemcor2twopunion} to obtain the following Lemma:
\begin{lemma}\label{lem contanddel}
Suppose that $\ga$ is a graph such that $\ga-e_{i}$, for some edge
$e_{i} \in \ee{\ga}$ with length $\li$ is connected. Then we have
\begin{equation*}
\begin{split}
\ta{\ga} = \frac{\li}{\li+\ri}\ta{\ga-e_{i}} +\frac{\ri}{\li+\ri}\ta{\oga_i}+\frac{\li^2-\li \ri}{12 (\li+\ri)}.
\end{split}
\end{equation*}
\end{lemma}
\begin{proof}
Multiply the formula in \lemref{lemcontract1} by $\frac{\ri}{\li+\ri}$ and the formula in \lemref{lemcor2twopunion} by $\frac{\li}{\li+\ri}$. Then add the results.
\end{proof}
To show the effect of edge deletion on the tau constant without using any terms with $A_{\pp,\qq,\ga-e_i}$, we have the following theorem:
\begin{theorem}\label{thm tau genus}
Let $\ga$ be a bridgeless graph with edges $\{ e_{1}, e_{2}, \dots,
e_{e} \}$. Then,
$$\tg=\frac{1}{g+1}\sum_{i=1}^{e}\frac{\li}{\li+\ri}\ta{\ga-e_{i}}+\frac{\ell(\ga)}{6(g+1)}-
\frac{r(\ga)}{4(g+1)}
.$$
\end{theorem}
\begin{proof}
Multiply both sides of the equation given in \lemref{lemcor2twopunion} by
$\frac{\li}{\li+\ri}$,
sum over all edges of $\ga$, and use the fact that
$\displaystyle \sum_{i=1}^{e}\frac{\li}{\li+\ri} = g$ (see \eqnref{eqn genus}) to obtain
\begin{equation*}
\begin{split}
g \cdot \ta{\ga} = \sum_{i=1}^{e} \frac{\li}{\li+\ri}\ta{\ga-e_{i}} +
\frac{z(\ga)}{12}-\frac{r(\ga)}{6}
+ \sum_{i=1}^{e}\frac{\li A_{p_{i},\qq,\ga-e_{i}}}{(\li+\ri)^2} .
\end{split}
\end{equation*}
Finally, we use \thmref{thmbasic2} and the fact that $z(\ga)+r(\ga)=\elg$ to complete the proof.
\end{proof}
As a corollary, we obtain a lower bound to the tau constant in terms of
the genus $g$.
\begin{corollary}\label{cor1 tau genus}
Let $\ga$ be a bridgeless metrized graph. Let edge $e_{i}$ have end points $\pp$ and $\qq$. For the
voltage function $j^{i}_{x}(\pp,\qq)$ on $\ga-e_{i}$, we have
\begin{equation*}
\begin{split}
\ta{\ga}= \frac{1}{4(g+1)}\sum_{e_i \in \ee{\ga}}
\frac{\li}{\li+\ri}\int_{\ga-e_{i}}(\ddx j^{i}_{x}(\pp,\qq))^2dx
+\frac{\ell(\ga)}{6(g+1)}.
\end{split}
\end{equation*}
In particular, $\tg \geq \frac{\ell(\ga)}{6(g+1)}$.
\end{corollary}
\begin{proof}
Applying \thmref{thmbasic} to $\ga-e_i$ gives
$\ta{\ga-e_{i}}=\frac{1}{4}\int_{\ga-e_{i}}(\ddx
j^{i}_{x}(\pp,\qq))^2dx+\frac{\ri}{4}$ for any edge $e_{i}$. Thus,
we obtain what we want by substituting this into \thmref{thm tau genus}.
\end{proof}
Using the results in this section, we attempted to apply induction arguments to prove
\conjref{TauBound} without a satisfactory outcome. We beleive that improving the results on $A_{\pp,\qq,\ga-e_i}$ will make the induction arguments applicable.
\section{Contraction, deletion and contraction-deletion identities}\label{section cont and del identities}

In this section, we will prove a number of identities, which we call ``contraction identities", ``deletion identities" and
``contraction-deletion" identities. These identities are interesting in their own right. One way to relate these identities to the tau constant can be explained as follows:

We know the exact values of the tau constant when the metrized graph is a tree or circle (see \lemref{lem tau for tree and circle}). If a metrized graph has vertex connectivity $1$ or $2$, we can express its tau constant in terms of the tau constants of its subgraphs (see \remref{remcutvertex1} and \cite[Theorems 5.1 and 8.1]{C2}).
After a sequence of edge deletions and contractions we can pass to these type of graphs from an arbitrary metrized graph. In the previous section, we gave formulas expressing $\tg$ in terms of $\ta{\ga-e_i}$'s or $\ta{\oga_i}$'s by considering all edge deletions or contractions of depth $1$ (see \thmref{thm tau contract} and \thmref{thm tau genus}). One wonders if it is possible to generalize these two theorems with further depths of edge deletions and contractions. The solution is given by the identities shown in this section. The identities of this section have crucial roles in generalizing the results of the previous section, as in the following section where we deal with the successive edge contractions.

Some of these ``contraction identities", ``deletion identities" and
``contraction-deletion" identities were proven in \cite[Sections 3.6 and 3.7]{C1}
using different methods. Our approach in this paper is to utilize Euler's formula for homogeneous functions as in the proof of \thmref{thmbasic2} in \cite{C2}.

Let $\ga$ be a graph with edges $\ee{\ga}=\{e_1, e_2, \dots,
e_e\}$, and let $\ga-e_i$ be the graph obtained by deleting the $i$-th edge
$e_{i} \in \ee{\ga}$. As before $\li$ is the length of edge $e_i$. Let $\ga^{DA}$ be the graph
obtained from $\ga$ by replacing each edge $e_i \in \ee{\ga}$ by $2$
edges $e_{i,1}, e_{i,2}$ of equal lengths
$\frac{\li}{2}$. Here DA stands for ``Double Adjusted". Then,
$\vv{\ga}=\vv{\ga^{DA}}$ and $\ell(\ga)=\ell(\ga^{DA})$.

Given a graph $\ga$, we will compare $\tau$-constants of the
following graphs: $\ga$, $\ga-e_i$, $\ga^{DA}$, $(\ga-e_i)^{DA}$,
$\ga^{DA}-e_{i,j}$ and $\ga^{DA}-\{e_{i,1}, \; e_{i,2}\}$.
It will turn out that by doing so, we will obtain non-trivial identities for
$\tg$, $z(\ga)$ and $r(\ga)$.

The graphs in Figure \ref{fig deletionid} illustrate what we will do. Graph $I$ shows
$\ga$ with an edge $e_i \in \ee{\ga}$ labeled by $i$. $II$ shows
$\ga-e_i$, $III$ shows $(\ga-e_i)^{DA}$, $IV$ shows $\ga^{DA}$ with
edges $e_{i,1}$ and $e_{i,2}$ labeled by $i$ and $ii$, $V$ shows
$\ga^{DA}-e_{i,2}$ and $VI$ shows $\ga^{DA}-\{e_{i,1}, e_{i,2}\}$.
\begin{figure}
\centerline{\epsffile{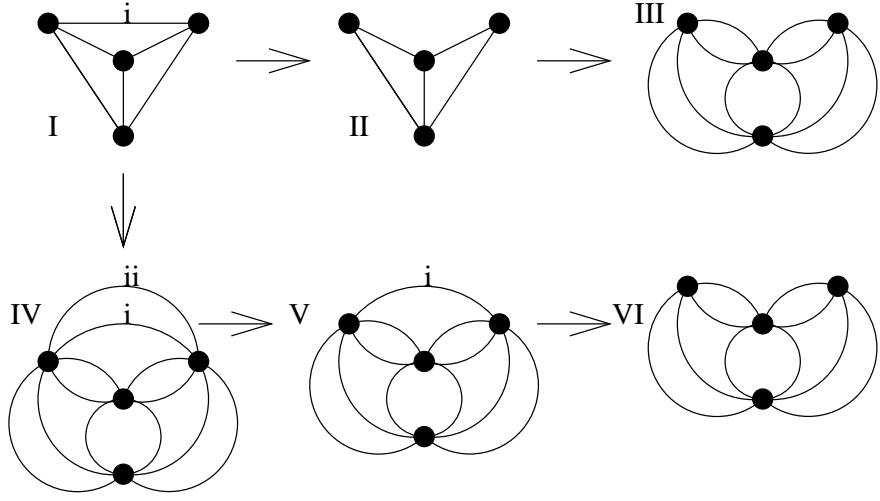}} \caption{Graphs used to obtain the deletion identities.} \label{fig deletionid}
\end{figure}

Note that $(\ga-e_i)^{DA}$ and
$\ga^{DA}-\{e_{i,1}, e_{i,2}\}$ are the same graphs. In \figref{fig deletionid}, they are the graphs in $III$ and $VI$.

Let $\ee{\ga}=\{e_1, \, e_2, \dots, e_e \}$, and let $\li$ be the
length of $e_i$. Then
\\$\ee{\ga^{DA}}=\{e_{1,1}, \, e_{1,2}, \,
e_{2,1}, \, e_{2,2}, \dots, e_{e,1}, \, e_{e,2}\}$. We write
$\ga^{d_i}:=\ga^{DA}-\{e_{i,1}, \, e_{i,2}\}$ to simplify the
notation.
\begin{theorem}\label{thm deletionid1}
Let $\ga$ be a bridgeless metrized graph. Given an edge $e_i \in \, \ee{\ga}$
with end points $\pp$ and $\qq$,
\[
\ta{\ga^{DA}}=\ta{(\ga-e_i)^{DA}}+\frac{2\li^2- \ri^2}{24(\li+\ri)}
+\frac{4}{\li+\ri}A_{\pp,\qq,\ga^{d_i}}.
\]
\end{theorem}
\begin{proof}
By applying \lemref{lemcor2twopunion} to $\ga^{DA}$ for the edge
$e_{i,1} \in \ee{\ga^{DA}}$ and using $R_{i,1}(\ga^{DA})=\frac{1}{2}
\frac{\li\ri}{2\li+\ri}$ from \cite[Lemma 3.10  with $n=2$]{C2}, we
get
\begin{equation}\label{eqndelid1}
\begin{split}
\ta{\ga^{DA}} = \ta{\ga^{DA}-e_{i,1}} +\frac{\li}{24}-\frac{1}{12}\frac{\li\ri}{2\li+\ri}+
\frac{A_{p_{i},\qq,\ga^{DA}-e_{i,1}}}{\frac{\li}{2}+\frac{1}{2}\frac{\li\ri}{2\li+\ri}}.
\end{split}
\end{equation}
By applying \cite[Lemma 8.6]{C2} to $A_{\pp,\qq,\ga^{DA}-e_{i,1}}$
with edge $e_{i,2}$, we obtain
\begin{equation}\label{eqndelid2}
\begin{split}
A_{\pp,\qq,\ga^{DA}-e_{i,1}}=\frac{4\li^2 A_{\pp,\qq,\ga^{d_i}}}{(2\li+\ri)^2}+\frac{1}{24}(\frac{\li \ri}{2\li+\ri})^2.
\end{split}
\end{equation}
Next, applying \lemref{lemcor2twopunion} to $\ga^{DA}-e_{i,1}$ with
respect to the edge $e_{i,2}$ and using
$R_{i,2}(\ga^{DA}-e_{i,1})=\frac{\ri}{4}$ gives
\begin{equation}\label{eqndelid3}
\begin{split}
\ta{\ga^{DA}-e_{i,1}}=\ta{\ga^{d_i}}+\frac{\li}{24}-\frac{\ri}{24}
+\frac{4A_{\pp,\qq,\ga^{d_i}}}{2\li+\ri}.
\end{split}
\end{equation}
We also note that $\ga^{d_i}=(\ga-e_i)^{DA}$.

Substituting Equations (\ref{eqndelid2}) and (\ref{eqndelid3}) into
\eqnref{eqndelid1} gives the result.
\end{proof}
\begin{notation} Let $\ga$ be a bridgeless metrized graph. Then for any $e_i \in \ee{\ga}$, we set
$$K_i(\ga):=\sum_{\substack{e_j \in \; \ee{\ga}\\ j\not=i}}\frac{L_j^2}{L_j+R_j}
-\sum_{e_j \in \; \ee{\ga-e_i}}\frac{L_j^2}{L_j+R_j(\ga-e_i)}.$$
\end{notation}
\begin{remark}\label{remdel1}
Let $\ga$ be a metrized graph and let $e_i \in \ee{\ga}$. For every $j \not=i$ and $j \in \{1, 2, \dots, e\}$,
$R_j(\ga-e_i) \geq R_j$ by Rayleigh's Cutting law, which states that cutting branches can only increase the effective resistance between any two points in a circuit (See \cite{DS} for more information). Therefore, $\frac{L_j^2}{L_j+R_j(\ga-e_i)} \leq
\frac{L_j^2}{L_j+R_j}$. Hence, $K_i(\ga) \geq 0$.
\end{remark}
\begin{theorem}\label{thm deletionid2}
Let $\ga$ be a bridgeless metrized graph. For any edge $e_i \in \, \ee{\ga}$ with end points $\pp$ and $\qq$,
\[
\frac{A_{\pp,\qq,\ga-e_i}}{\li+\ri}=\frac{16A_{\pp,\qq,\ga^{d_i}}}{\li+\ri}-\frac{K_i(\ga)}{6}.
\]
\end{theorem}
\begin{proof}
Note that $\ell(\ga^{DA})=\ell(\ga)$. Applying \cite[Corollary 3.5]{C2} to $\ga^{DA}$, we obtain
\begin{equation}\label{eqndelid4}
\ta{\ga^{DA}}=\frac{\ell(\ga)}{48} + \frac{\ta{\ga}}{4}+\frac{z(\ga)}{24}.
\end{equation}
Applying \cite[Corollary 3.5]{C2} to $(\ga-e_i)^{DA}$, we obtain
\begin{equation}\label{eqndelid5}
\ta{(\ga-e_i)^{DA}}=\frac{\ell(\ga-e_i)}{48} + \frac{\ta{\ga-e_i}}{4}+\frac{z(\ga-e_i)}{24}.
\end{equation}
Substituting \eqnref{eqndelid4} and \eqnref{eqndelid5} into
\thmref{thm deletionid1}, and recalling that
$\ell(\ga-e_i)=\ell(\ga)-\li$ gives
\begin{equation}\label{eqndelid6}
\begin{split}
\tg =\ta{\ga-e_i}+\frac{\li}{12}-\frac{\ri}{6}-\frac{K_i(\ga)}{6}+\frac{16A_{\pp,\qq,\ga^{d_i}}}{\li+\ri}.
\end{split}
\end{equation}
Comparing \eqnref{eqndelid6} with \lemref{lemcor2twopunion} gives the result.
\end{proof}
Let $p$, $q$  be any two points in $\ga$, and let $e_{0}$
be a line segment of length $L$. By identifying the end points of
$e_{0}$ with p and q of $\ga$ we obtain a new graph which we
denote by $\ga_{(p,q)}$. Then
$\ell(\ga_{(p,q)})=\ell(\ga)+L$. Also, by identifying p and q
with each other in $\ga$ we obtain a graph which we denote by
$\ga_{pq}$. Then $\ell(\ga_{pq})=\ell(\ga)$. If $p$ and $q$ are end
points of an edge $e_i \in \ga$, then $\ga_{pq}=\tga_i$.
\begin{lemma}\cite[Corollaries 7.1 and 7.2]{C2}\label{lemcoradding1}
Let $\ga$ be a metrized graph with resistance function
$r(x,y)$. For $p$, $q$, $\ga_{(p,q)}$, and $\ga_{pq}$ as given above,
$$\ta{\ga_{(p,q)}}=\tg+\frac{L}{12}-\frac{r(p,q)}{6}+\frac{A_{p,q,\ga}}{L+r(p,q)},
\qquad \ta{\ga_{pq}}=\tg-\frac{r(p,q)}{6}+\frac{A_{p,q,\ga}}{r(p,q)}.$$
\end{lemma}
The following corollary is the initial step towards the contraction-deletion identities (Theorems \ref{thm contraction-deletion1} and \ref{thm cont-del for z(ga)}).
\begin{corollary}\label{cor deletionid2}
Let $\ga$ be a metrized graph with resistance function
$r(x,y)$, and let $p$, $q$, $e_0$ and $\ga_{(p,q)}$ be as above. Corresponding to the
edge $e_0$, suppose that we have the pair of edges $e_{0,1}$ and $e_{0,2}$ in $\ee{(\ga_{(p,q)})^{DA}}$. Then we have
$$\frac{A_{p,q,\ga}}{L+r(p,q)}=\frac{16 A_{p,q,\ga^{DA}}}{L+r(p,q)}-\frac{1}{6}
\Big(\sum_{\substack{e_j \in \; \ee{\ga_{(p,q)}}\\e_j\not=e_0}}\frac{L_j^2}{L_j+R_j(\ga_{(p,q)})}
-\sum_{e_i \in \; \ee{\ga}}\frac{L_i^2}{L_i+R_i} \Big).$$
\end{corollary}
\begin{proof}
\thmref{thm deletionid2} applied to $\ga_{(p,q)}$ with edge $e_0$ gives
$\frac{A_{p,q,\ga_{(p,q)}-e_0}}{L+r(p,q)}=\frac{16 A_{p,q,(\ga_{(p,q)})^{DA}-\{e_{0,1}, \, e_{0,2}\}}}{L+r(p,q)}-\frac{1}{6}
\Big(\sum_{\substack{e_j \in \; \ee{\ga_{(p,q)}}\\e_j\not=e_0}}\frac{L_j^2}{L_j+R_j(\ga_{(p,q)})}
-\sum_{e_i \in \; \ee{\ga_{(p,q)}-e_0}}\frac{L_i^2}{L_i+R_i} \Big).$
On the other hand, we have $\ga_{(p,q)}-e_0=\ga$ and $(\ga_{(p,q)})^{DA}-\{e_{0,1}, \, e_{0,2}\}=\ga^{DA}$. This gives the result.
\end{proof}
Let $\ga'_{(p,q)}$ be a metrized graph obtained from $\ga$ by connecting the points $p$ and $q$ of $\ga$ with line segment $e_0'$ of length $L'$. Then, $\ga_{(p,q)}-e_0=\ga'_{(p,q)}-e_0'$. Let $L=t_1 \cdot r(p,q)$ and $L'=t_2 \cdot r(p,q)$ for some positive real numbers $t_1$ and $t_2$. By applying \corref{cor deletionid2} to $\ga_{(p,q)}$ and $\ga'_{(p,q)}$, we obtain
\begin{equation}\label{eqn cor deletionid2}
\begin{split}
&(1+t_1) \Big(z(\ga_{(p,q)})-\frac{L^2}{L+r(p,q)}-z(\ga) \Big) =
(1+t_2) \Big(z(\ga'_{(p,q)})-\frac{(L')^2}{L'+r(p,q)}-z(\ga) \Big).
\end{split}
\end{equation}
As $t_2 \rightarrow 0$, we have $L'\rightarrow 0$ and $\ga'_{(p,q)} \rightarrow \ga_{pq}$, and
so $z(\ga'_{(p,q)})\rightarrow z(\ga_{pq})$. We substitute $t_1=\frac{L}{r(p,q)}$ into \eqnref{eqn cor deletionid2}. Then we obtain the following relation as $t_2 \rightarrow 0$:
\begin{equation}\label{eqn cor deletionid3}
\begin{split}
z(\ga_{(p,q)})=\frac{L^2}{L+r(p,q)}+\frac{L}{L+r(p,q)}z(\ga)+\frac{r(p,q)}{L+r(p,q)}z(\ga_{pq}).
\end{split}
\end{equation}
We use \eqnref{eqn cor deletionid3} to obtain the following Theorem:
\begin{theorem}\label{thm cont-del for z(ga)}
Let $\ga$ be a metrized graph. For each edge $e_i \in \ee{\ga}$ such that $\ga-e_i$ is connected, we have
$$z(\ga)=\frac{\li^2}{\li+\ri}+\frac{\li}{\li+\ri}z(\ga-e_i)+\frac{\ri}{\li+\ri}z(\oga_i).$$
 \end{theorem}
\begin{proof}
In \eqnref{eqn cor deletionid3}, replace $\ga_{(p,q)}$ by $\ga$, $L$ by $\li$, $\ga$ by $\ga-e_i$.
This gives what we wanted to show.
\end{proof}
We call the identity in \thmref{thm cont-del for z(ga)} the contraction-deletion identity for $z(\ga)$.

If $e_i$ is a bridge (i.e., $\ri=\infty$), $z(\ga)=z(\oga_i)$, which can also be seen from \thmref{thm cont-del for z(ga)} as $\ri\rightarrow \infty$.

Moreover, for any metrized graph $\ga$ and for each edge $e_i \in \ee{\ga}$ such that $\ga-e_i$ is connected, we obtain the expression below for $K_i(\ga)$ by using its definition and \thmref{thm cont-del for z(ga)}:
\begin{equation}\label{eqn cor deletionid4}
\begin{split}
K_i(\ga)=\frac{\ri}{\li+\ri}\Big(z(\oga_i)-z(\ga-e_i) \Big).
\end{split}
\end{equation}

A function $f: \RR^n \rightarrow \RR$ is called homogeneous of degree $k$ if $f(\lambda
x_1,\lambda x_2, \cdots, \lambda x_n)=\lambda^k f(x_1,x_2, \cdots,
x_n)$ for $\lambda > 0$. A continuously differentiable function $f: \RR^n \rightarrow
\RR$ which is homogeneous of degree $k$ has the following property:
\begin{equation}\label{eqn Euler}
k \cdot f=\sum_{i=1}^n x_i \frac{\partial f}{\partial x_i} .
\end{equation}
Equation (\ref{eqn Euler}) is called Euler's formula.

For a given metrized graph $\ga$ with $\# (\ee{\ga})=e$, let $\{L_1, L_2,
\cdots, L_e \}$ be the edge lengths. Then $z:\RR^e_{>0} \rightarrow
\RR$ given by $z(L_1, L_2, \cdots, L_e)=z(\ga)$ is a continuously
differentiable homogeneous function of degree $1$, when we consider
all possible length distributions without changing the topology of
the graph $\ga$.
\begin{lemma}\label{lem diff2}
Let $\ga$ be a metrized graph, and let $e_i \in \ee{\ga}$ be of length $L_i$ such that $\ga-e_i$ is connected. Then we have
$$\frac{\partial z(\ga)}{\partial L_i} =\frac{\li (\li +2 \ri)}{(\li+\ri)^2}+\frac{\ri}{(\li+\ri)^2}z(\ga-e_i)-\frac{\ri}{(\li+\ri)^2}z(\oga_i).$$
\end{lemma}
\begin{proof}
Note that $z(\oga_i)$, $z(\ga-e_i)$ and $\ri$ are independent of $\li$. Thus, taking the partial derivatives of the both sides of the identity in \thmref{thm cont-del for z(ga)} with respect to $\li$ gives the result.
\end{proof}
\begin{theorem}\label{thm diff1}
Let $\ga$ be a bridgeless metrized graph. Then we have
$$\sum_{e_i \in \ee{\ga}}\frac{\li K_i(\ga)}{\li+\ri}= \sum_{e_i \in \ee{\ga}}\frac{\li \ri}{(\li+\ri)^2}
\big(z(\oga_i)-z(\ga-e_i) \big)=\sum_{e_i \in \ee{\ga}}\frac{\li^2 \ri}{(\li+\ri)^2}.$$
\end{theorem}
\begin{proof}
The first equality follows from \eqnref{eqn cor deletionid4}. By Euler's formula, $z(\ga)=\sum_{e_i \in \ee{\ga}}\li \cdot \frac{\partial z(\ga)}{\partial L_i}$. Then the second equality follows from \lemref{lem diff2}.
\end{proof}
For a given metrized graph $\ga$ with $\# (\ee{\ga})=e$, let $\{L_1, L_2,
\cdots, L_e \}$ be the edge lengths. Then $r:\RR^e_{>0} \rightarrow
\RR$ given by $r(L_1, L_2, \cdots, L_e)=r(\ga)$ is a continuously
differentiable homogeneous function of degree $1$, when we consider
all possible length distributions without changing the topology of
the graph $\ga$.
\begin{lemma}\label{lem diff3}
Let $\ga$ be a metrized graph, and let $e_i \in \ee{\ga}$ be of length $L_i$ such that $\ga-e_i$ is connected. Then we have
$$\frac{\partial r(\ga)}{\partial L_i} =\frac{\ri^2}{(\li+\ri)^2}+\frac{\ri}{(\li+\ri)^2}r(\ga-e_i)-\frac{\ri}{(\li+\ri)^2}r(\oga_i).$$
\end{lemma}
\begin{proof}
Since $\ell(\ga)=z(\ga)+r(\ga)$ for any graph, and $\ell(\ga-e_i)=\ell(\oga_i)=\ell(\ga)-\li$, \thmref{thm cont-del for z(ga)} is equivalent to
\begin{equation}\label{eqn cor deletionid5}
\begin{split}
r(\ga)=\frac{\li \ri}{\li+\ri}+\frac{\li}{\li+\ri}r(\ga-e_i)+\frac{\ri}{\li+\ri}r(\oga_i).
\end{split}
\end{equation}
Note that $r(\oga_i)$, $r(\ga-e_i)$ and $\ri$ are independent of $\li$. Thus, taking the partial derivatives of the both sides of \eqnref{eqn cor deletionid5} with respect to $\li$ gives the result.
\end{proof}
Let $\ga-e_i$ be a connected graph for an edge $e_i \in \ee{\ga}$ of length $\li$. Suppose $\pp$ and $\qq$ are the end points of $e_i$, and $p \in \ga-e_i$. By applying circuit reductions,
$\ga-e_i$ can be transformed into a $Y$-shaped graph with the same resistances between $\pp$, $\qq$, and $p$ as in $\ga-e_i$ (see the articles \cite{CR} and \cite[Section $2$]{C2}).
The resulting graph is shown by the first graph in \figref{fig edgedelete5}, with the corresponding voltage values on each segment, where $\hj{x}{y}{z}$ is the voltage function in $\ga-e_i$. Since $\ga-e_i$ has such circuit reduction, $\ga$ has the circuit reduction as the second graph in \figref{fig edgedelete5}. From now on, we will use the following notation:
$R_{a_i,p} := \hj{\pp}{p}{\qq}$, $R_{b_i,p} := \hj{\qq}{\pp}{p}$, $R_{c_i,p} := \hj{p}{\pp}{\qq}$. Let $\ri$ be the resistance
between $\pp$ and $\qq$ in $\ga-e_i$. Note that $R_{a_i,p}+R_{b_i,p}=\ri$ for each $p \in \ga$.

If $\ga-e_i$ is not connected, we set $R_{b_i,p}=\ri=\infty$ and $R_{a_i,p}=0$ if $p$ belongs to the component of $\ga-e_i$
containing $\pp$, and we set $R_{a_i,p}=\ri=\infty$ and $R_{b_i,p}=0$ if $p$ belongs to the component of $\ga-e_i$
containing $\qq$.
\begin{figure}
\centering
\includegraphics[scale=0.7]{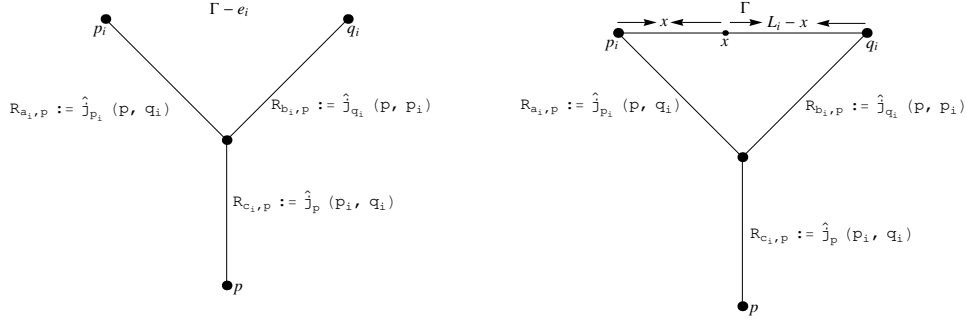} \caption{Circuit reduction of $\ga-e_i$ with reference to $\pp$, $\qq$ and $p$.} \label{fig edgedelete5}
\end{figure}

In the rest of the paper, for any metrized graph $\ga$ and a fixed vertex $p \in \vv{\ga}$ we will use the following notation:
\begin{equation*}
\begin{split}
y(\ga)&=\frac{1}{4}\sum_{e_i \, \in \ee{\ga}}\frac{\li
\ri^2}{(\li+\ri)^2}+\frac{3}{4}\sum_{e_i \, \in \ee{\ga}}\frac{\li
(R_{a_i,p}-R_{b_i,p})^2}{(\li+\ri)^2},
\\x(\ga)&=\sum_{e_i \, \in
\ee{\ga}}\frac{\li^2\ri}{(\li+\ri)^2} +\frac{3}{4}\sum_{e_i \, \in
\ee{\ga}}\frac{\li \ri^2}{(\li+\ri)^2}
-\frac{3}{4}\sum_{e_i \, \in \ee{\ga}}\frac{\li
(R_{a_i,p}-R_{b_i,p})^2}{(\li+\ri)^2}.
\end{split}
\end{equation*}
If $\ga-e_i$ is not connected for an edge $e_i$, i.e. $\ri$ is infinite (and $(R_{a_i,p}-R_{b_i,p})^2=\ri^2$), the
summands should be considered as their corresponding limits as $\ri\longrightarrow \infty$.

It follows from \cite[Proposition 2.9]{C2} that
\begin{equation}\label{eqn tau x and y}
\begin{split}
\tg=\frac{\elg}{12}-\frac{x(\ga)}{6}+\frac{y(\ga)}{6}.
\end{split}
\end{equation}
It is easy to see that
\begin{equation}\label{eqn sumof x and y}
\begin{split}
r(\ga)=x(\ga)+y(\ga), \qquad \text{and so} \qquad \elg=x(\ga)+y(\ga)+z(\ga).
\end{split}
\end{equation}
We call the following identities the contraction-deletion identities for $x(\ga)$ and $y(\ga)$.
\begin{theorem}\label{thm contraction-deletion1}
Let $\ga$ be a metrized graph with an edge $e_i \in \ee{\ga}$ such that $\ga-e_i$ is connected. Then we have
\begin{equation*}
\begin{split}
x(\ga) &=\frac{\li \ri}{\li+\ri}+\frac{\li}{\li+\ri}x(\ga-e_i)+\frac{\ri}{\li+\ri}x(\oga_i),
\\ y(\ga)&=\frac{\li}{\li+\ri}y(\ga-e_i)+\frac{\ri}{\li+\ri}y(\oga_i).
\end{split}
\end{equation*}
\end{theorem}
\begin{proof}
By Equations (\ref{eqn cor deletionid5}) and (\ref{eqn sumof x and y}),
\begin{equation}\label{eqn cor deletionid5b}
\begin{split}
x(\ga)+y(\ga)=\frac{\li \ri}{\li+\ri}+\frac{\li}{\li+\ri}\big(x(\ga-e_i)+ y(\ga-e_i)\big)+\frac{\ri}{\li+\ri}\big(x(\oga_i)+y(\oga_i) \big).
\end{split}
\end{equation}
 On the other hand, by \lemref{lem contanddel} and \eqnref{eqn tau x and y} applied to each of $\ga$, $\ga-e_i$ and $\oga_i$  we have
\begin{equation}\label{eqn cor deletionid5c}
\begin{split}
x(\ga)-y(\ga) &=\frac{\li \ri}{\li+\ri}+\frac{\li}{\li+\ri}\big(x(\ga-e_i)- y(\ga-e_i)\big)
+ \frac{\ri}{\li+\ri}\big(x(\oga_i)-y(\oga_i) \big).
\end{split}
\end{equation}
Hence, the result follows from \eqnref{eqn cor deletionid5b} and \eqnref{eqn cor deletionid5c}.
\end{proof}
\begin{lemma}\label{lem contraction Apq}
Let $\ga$ be a metrized graph with an edge $e_i \in \ee{\ga}$ such that $\ga-e_i$ is connected. Let $\pp$ and $\qq$
be end points of $e_i$. Then we have
$$x(\ga)-y(\ga)=x(\oga_i)-y(\oga_i)+6\frac{\li A_{\pp,\qq,\ga-{e_i}}}{\ri (\li+\ri)}.$$
\end{lemma}
\begin{proof}
It follows from \lemref{lemcontract1} and \lemref{lemcor2twopunion} that
\begin{equation}\label{eqn contr Apq}
\begin{split}
\ta{\oga_i}=\ta{\ga-e_{i}} -\frac{\ri}{6}+\frac{A_{p_{i},\qq,\ga-e_{i}}}{\ri}.
\end{split}
\end{equation}
From \eqnref{eqn contr Apq} and \eqnref{eqn tau x and y} applied to both $\oga_i$ and $\ga-e_i$, we get
\begin{equation}\label{eqn contr Apq2}
\begin{split}
x(\oga_i)-y(\oga_i)=x(\ga-e_{i})-y(\ga-e_{i})+\ri-6 \frac{A_{p_{i},\qq,\ga-e_{i}}}{\ri}.
\end{split}
\end{equation}
Therefore, we obtain the result by solving \eqnref{eqn contr Apq2} for $x(\ga-e_{i})-y(\ga-e_{i})$ and substituting
into \eqnref{eqn cor deletionid5c}.
\end{proof}
For a given metrized graph $\ga$ with $\# (\ee{\ga})=e$, let $\{L_1, L_2,
\cdots, L_e \}$ be the edge lengths. Both of the functions $x:\RR^e_{>0} \rightarrow
\RR$ given by $x(L_1, L_2, \cdots, L_e)=x(\ga)$ and $y:\RR^e_{>0} \rightarrow
\RR$ given by $y(L_1, L_2, \cdots, L_e)=y(\ga)$  are continuously
differentiable homogeneous functions of degree $1$, when we consider
all possible length distributions without changing the topology of
$\ga$.
\begin{theorem}\label{thm contraction-deletion2}
Let $\ga$ be a bridgeless metrized graph. Then we have
\begin{equation*}
\begin{split}
x(\ga) &=\sum_{e_i \in \ee{\ga}}\frac{\li \ri^2}{(\li+\ri)^2}+\sum_{e_i \in \ee{\ga}}\frac{\li \ri}{(\li+\ri)^2}\big(x(\ga-e_i)-x(\oga_i)\big),
\\ y(\ga)&=\sum_{e_i \in \ee{\ga}}\frac{\li \ri}{(\li+\ri)^2}\big(y(\ga-e_i)-y(\oga_i)\big).
\end{split}
\end{equation*}
\end{theorem}
\begin{proof}
By taking the partial derivatives of the both sides of the equalities in \thmref{thm contraction-deletion1} with respect to $\li$ gives
\begin{equation}\label{eqn partialsof x and y}
\begin{split}
\frac{\partial x(\ga)}{\partial L_i} &=\frac{\ri^2}{(\li+\ri)^2}+\frac{\ri}{(\li+\ri)^2}x(\ga-e_i)-\frac{\ri}{(\li+\ri)^2}x(\oga_i),
\\
\frac{\partial y(\ga)}{\partial L_i} &=\frac{\ri}{(\li+\ri)^2}y(\ga-e_i)-\frac{\ri}{(\li+\ri)^2}y(\oga_i).
\end{split}
\end{equation}
Therefore, by applying Euler's formula we obtain the equalities we wanted.
\end{proof}
We call the following identities the contraction identities for $x(\ga)$ and $y(\ga)$.
\begin{theorem}\label{thm contraction}
Let $\ga$ be a bridgeless metrized graph with $v=\#(\vv{\ga})\geq 2$. Then we have
%
\begin{align*}
(v-2) x(\ga) &=\sum_{e_i \in \ee{\ga}}\frac{\ri}{\li+\ri}x(\oga_i), & \qquad
(v-2) y(\ga) &=\sum_{e_i \in \ee{\ga}}\frac{\ri}{\li+\ri}y(\oga_i),
\end{align*}
\end{theorem}
\begin{proof}
Multiplying both sides of the equalities in \thmref{thm contraction-deletion1} by $\frac{\ri}{\li+\ri}$, and using the fact that
$\sum_{e_i \in \ee{\ga}}\frac{\ri}{\li+\ri}=v-1$ (see \eqnref{eqn genus}) we obtain
\begin{equation}\label{eqn contractionid1}
\begin{split}
(v-1) x(\ga) &=\sum_{e_i \in \ee{\ga}}\frac{\li \ri^2}{(\li+\ri)^2}+\sum_{e_i \in \ee{\ga}}\frac{\li \ri}{(\li+\ri)^2}x(\ga-e_i)+
\sum_{e_i \in \ee{\ga}}\frac{\ri^2}{(\li+\ri)^2} x(\oga_i),
\\ (v-1) y(\ga) &=\sum_{e_i \in \ee{\ga}}\frac{\li \ri}{(\li+\ri)^2}y(\ga-e_i)+
\sum_{e_i \in \ee{\ga}}\frac{\ri^2}{(\li+\ri)^2} y(\oga_i).
\end{split}
\end{equation}
Thus, the result follows from \eqnref{eqn contractionid1} and \thmref{thm contraction-deletion2}.
\end{proof}
We call the first identity in the corollary below the contraction identity for $z(\ga)$.
\begin{corollary}\label{cor contraction1}
Let $\ga$ be a bridgeless metrized graph with $v=\#(\vv{\ga})\geq 2$. Then we have
\begin{equation*}
\begin{split}
(v-1) z(\ga) =\sum_{e_i \in \ee{\ga}}\frac{\ri}{\li+\ri}z(\oga_i), \qquad
(v-2) r(\ga) &=\sum_{e_i \in \ee{\ga}}\frac{\ri}{\li+\ri}r(\oga_i).
\end{split}
\end{equation*}
 \end{corollary}
\begin{proof}
The second equality follows by adding the expressions in \thmref{thm contraction} and using the fact that
$\elg=z(\ga)+r(\ga)$. Using the second equality along with the facts that $z(\oga_i)=\elg-\li-r(\oga_i)$ and
$\displaystyle \sum_{e_i \in \ee{\ga}}\frac{\ri}{\li+\ri}=v-1$ (see \eqnref{eqn genus}), we obtain the first equality.
\end{proof}
\begin{corollary}\label{cor contraction2}
Let $\ga$ be a bridgeless metrized graph with $v=\#(\vv{\ga})\geq 3$. Then we have
\begin{equation*}
\begin{split}
\tg =\frac{\elg}{12}-\frac{1}{6 (v-2)}\sum_{e_i \in \ee{\ga}}\frac{\ri}{\li+\ri}\big(x(\oga_i)-y(\oga_i)\big).
\end{split}
\end{equation*}
 \end{corollary}
\begin{proof}
By \thmref{thm contraction}, we have
\begin{equation}\label{eqn contractionid2}
\begin{split}
(v-2) (x(\ga)-y(\ga))=\sum_{e_i \in \ee{\ga}}\frac{\ri}{\li+\ri}\big(x(\oga_i)-y(\oga_i)\big).
\end{split}
\end{equation}
Thus, the result follows from \eqnref{eqn tau x and y}.
\end{proof}
We call the identities in \thmref{thm deletion} and \corref{cor deletion1} the deletion identities.
\begin{theorem}\label{thm deletion}
Let $\ga$ be a bridgeless metrized graph. Then we have
\begin{equation*}
\begin{split}
g \cdot x(\ga)=
y(\ga)+\sum_{e_i \in \ee{\ga}}\frac{\li}{\li+\ri}x(\ga-e_i),
\qquad (g+1) y(\ga) &=\sum_{e_i \in \ee{\ga}}\frac{\li}{\li+\ri}y(\ga-e_i).
\end{split}
\end{equation*}
\end{theorem}
\begin{proof}
Multiplying both sides of the equalities in \thmref{thm contraction-deletion1} by $\frac{\li}{\li+\ri}$, and using the fact that
$\sum_{e_i \in \ee{\ga}}\frac{\li}{\li+\ri}=g$ (see \eqnref{eqn genus}) we obtain
\begin{equation}\label{eqn deletionid1}
\begin{split}
g \cdot x(\ga) &=\sum_{e_i \in \ee{\ga}}\frac{\li^2 \ri}{(\li+\ri)^2}+\sum_{e_i \in \ee{\ga}}\frac{\li^2}{(\li+\ri)^2}x(\ga-e_i)+
\sum_{e_i \in \ee{\ga}}\frac{\li \ri}{(\li+\ri)^2} x(\oga_i),
\\ g \cdot y(\ga) &=\sum_{e_i \in \ee{\ga}}\frac{\li^2}{(\li+\ri)^2}y(\ga-e_i)+
\sum_{e_i \in \ee{\ga}}\frac{\li \ri}{(\li+\ri)^2} y(\oga_i).
\end{split}
\end{equation}

The first equality is obtained by adding the first equalities in \thmref{thm contraction-deletion2} and \eqnref{eqn deletionid1} and using the fact that $r(\ga)=x(\ga)+y(\ga)$.

Similarly, the second equality is obtained by adding the second equalities in \thmref{thm contraction-deletion2} and \eqnref{eqn deletionid1}.
\end{proof}
\begin{corollary}\label{cor deletion1}
Let $\ga$ be a bridgeless metrized graph. Then we have
\begin{equation*}
\begin{split}
(g-1) z(\ga) =\sum_{e_i \in \ee{\ga}}\frac{\li}{\li+\ri}z(\ga-e_i), \qquad
g \cdot r(\ga) &=\sum_{e_i \in \ee{\ga}}\frac{\li}{\li+\ri}r(\ga-e_i).
\end{split}
\end{equation*}
 \end{corollary}
\begin{proof}
Adding the identities in \thmref{thm deletion} and using the fact that $r(\ga)=x(\ga)+y(\ga)$ give the second formula.

Then the first formula is obtained by using the second formula, \eqnref{eqn genus} and the fact that $\elg=z(\ga)+r(\ga)$.
\end{proof}
\begin{corollary}\label{cor deletion2}
Let $\ga$ be a bridgeless metrized graph. Then we have
\begin{equation*}
\begin{split}
\tg =& \frac{\elg}{12}-\frac{1}{6 (g+1)}\sum_{e_i \in \ee{\ga}}\frac{\li}{\li+\ri}\big(x(\ga-e_i)-y(\ga-e_i)\big)
\\ & \quad -\frac{1}{6 (g+1) g}\sum_{e_i \in \ee{\ga}}\frac{\li}{\li+\ri}r(\ga-e_i).
\end{split}
\end{equation*}
 \end{corollary}
\begin{proof}
By \thmref{thm deletion} and the fact that $r(\ga)=x(\ga)+y(\ga)$, we have
\begin{equation}\label{eqn delid2}
\begin{split}
(g+1) \cdot (x(\ga)-y(\ga))=r(\ga)+\sum_{e_i \in \ee{\ga}}\frac{\li}{\li+\ri}\big(x(\ga-e_i)-y(\ga-e_i)\big).
\end{split}
\end{equation}
Thus, the result follows from \eqnref{eqn tau x and y} and the second identity in \corref{cor deletion1}.
\end{proof}
In this section, we proved the following identities among other things:

By \thmref{thm contraction-deletion1} and \thmref{thm cont-del for z(ga)}, the contraction-deletion identities for a metrized graph $\ga$ and for an edge $e_i \in \ee{\ga}$ with connected $\ga-e_i$ are
\begin{equation}\label{eqn contraction-deletion identities}
\begin{split}
x(\ga) &=\frac{\li \ri}{\li+\ri}+\frac{\li}{\li+\ri}x(\ga-e_i)+\frac{\ri}{\li+\ri}x(\oga_i),
\\ y(\ga) &=\frac{\li}{\li+\ri}y(\ga-e_i)+\frac{\ri}{\li+\ri}y(\oga_i),
\\ z(\ga) &=\frac{\li^2}{\li+\ri}+\frac{\li}{\li+\ri}z(\ga-e_i)+\frac{\ri}{\li+\ri}z(\oga_i).
\end{split}
\end{equation}
By \thmref{thm contraction} and \corref{cor contraction1}, the contraction identities for a bridgeless metrized graph with $v=\#(\vv{\ga})\geq 2$ are
\begin{equation}\label{eqn contraction identities}
\begin{aligned}
(v-2) x(\ga) &=\sum_{e_i \in \ee{\ga}}\frac{\ri}{\li+\ri}x(\oga_i), & \qquad
(v-2) y(\ga) &=\sum_{e_i \in \ee{\ga}}\frac{\ri}{\li+\ri}y(\oga_i),
\\ (v-1) z(\ga) &=\sum_{e_i \in \ee{\ga}}\frac{\ri}{\li+\ri}z(\oga_i), & \qquad
(v-2) r(\ga) &=\sum_{e_i \in \ee{\ga}}\frac{\ri}{\li+\ri}r(\oga_i).
\end{aligned}
\end{equation}
By \thmref{thm deletion} and \corref{cor deletion1}, the deletion identities for a bridgeless $\ga$ are
\begin{equation}\label{eqn deletion identities}
\begin{aligned}
& g \cdot x(\ga) =y(\ga)+\sum_{e_i \in \ee{\ga}}\frac{\li}{\li+\ri}x(\ga-e_i),  & \:
& (g+1) y(\ga) =\sum_{e_i \in \ee{\ga}}\frac{\li}{\li+\ri}y(\ga-e_i),
\\ & (g-1) z(\ga) =\sum_{e_i \in \ee{\ga}}\frac{\li}{\li+\ri}z(\ga-e_i), & \:
& g \cdot r(\ga) =\sum_{e_i \in \ee{\ga}}\frac{\li}{\li+\ri}r(\ga-e_i).
\end{aligned}
\end{equation}
Also, for a bridgeless $\ga$  the following identity of \thmref{thm diff1} deserves attention:
\begin{equation}\label{eqn deserve attention}
\sum_{e_i \in \ee{\ga}}\frac{\li K_i(\ga)}{\li+\ri}= \sum_{e_i \in \ee{\ga}}\frac{\li \ri}{(\li+\ri)^2}
\big(z(\oga_i)-z(\ga-e_i) \big)=\sum_{e_i \in \ee{\ga}}\frac{\li^2 \ri}{(\li+\ri)^2}.
\end{equation}

\section{Successive edge contraction}\label{section successive edge
contraction}

In this section, we will successively contract edges in $\ee{\ga}$
for any metrized graph $\ga$. The contraction identities developed in the previous section will enable us to generalize the results of \secref{sec edge cont and del} and some of the results of \secref{section cont and del identities}. The results of this section will help us to understand the effects
of topological properties of $\ga$, such as the edge connectivity, on $\tg$.

Let $\ga$ be a metrized graph and let $\oga_i$ be the metrized graph obtained by contracting $i$-th edge
$e_{i} \in \ee{\ga}$ to its end points.
Similarly, for any integer $k \geq 2$, let $\oga_{i_1,i_2, \dots, i_k}$ be the metrized graph obtained by contracting $i_k$-th edge
$e_{i_k} \in \ee{\oga_{i_1,i_2, \dots, i_{k-1}}}$ to its end points. Note that
$\ee{\oga_{i_1,i_2, \dots, i_k}}= \ee{\ga}-\{e_{i_1}, e_{i_2}, \dots, e_{i_k}\}$ for any $k$.
Let $\oga_{i_0}:=\ga$.

Let $e_{i_k} \in \ee{\ga}$ be an edge of index $i_k$. Recall that we denote the resistance between the end points of $e_{i_k}$ in $\ga-e_{i_k}$ by $R_{i_k}$ and that we use $L_{i_k}$ to denote the length of $e_{i_k}$.

Now, we generalize Equation (\ref{eqn contractionid2}) as follows:
\begin{lemma}\label{lem succesive contraction of x-y}
Let $\ga$ be a bridgeless metrized graph with $(k+2) \leq v=\#(\vv{\ga)}$ for
some integer $k \geq 1$. Then
\begin{equation*}
\begin{split}
\frac{(v-2)!}{(v-k-2)!}\big(x(\ga)-y(\ga)\big) = & \sum_{e_{i_1} \in \;
\ee{\ga}}\frac{R_{i_1}}{L_{i_1}+R_{i_1}} \sum_{\substack{e_{i_2} \in
\\ \ee{\oga_{i_1}}}}\frac{R_{i_2}}{L_{i_2}+R_{i_2}}
\; \dots
\\ & \qquad \sum_{e_{i_k} \in \; \ee{\oga_{i_1, \dots, i_{k-1}}}}
\frac{R_{i_k}}{L_{i_k}+R_{i_k}}\big(x(\oga_{i_1,\dots, i_k})-y(\oga_{i_1,\dots, i_k})\big).
\end{split}
\end{equation*}
\end{lemma}
\begin{proof}
Note that if an edge of a bridgeless graph is contracted the resulting graph will be also bridgeless.
If an edge $e_{i_j}$ is a self loop, then $\frac{R_{i_j}}{L_{i_j}+R_{i_j}}=0$. Thus, contraction of self loops
does not contribute to sums in contraction identities.
Hence, we can inductively apply \eqnref{eqn contractionid2} to obtain the result.
\end{proof}
\begin{remark}
After contracting edges in a graph $\ga$, multiple edges or self-loops may appear. However, this does not cause any problem
for contraction identities.
\end{remark}
We can generalize \corref{cor contraction2} as follows:
\begin{theorem}\label{thm succesive contraction of x-y}
Let $\ga$ be a bridgeless metrized graph with $(k+2) \leq v=\#(\vv{\ga)}$ for
some integer $k \geq 1$. Then we have
\begin{equation*}
\begin{split}
\tg= &\frac{\elg}{12}-\frac{(v-k-2)!}{6 (v-2)!}\sum_{e_{i_1} \in \;
\ee{\ga}}\frac{R_{i_1}}{L_{i_1}+R_{i_1}} \sum_{\substack{e_{i_2} \in
\\ \ee{\oga_{i_1}}}}\frac{R_{i_2}}{L_{i_2}+R_{i_2}}
\; \dots
\\ & \qquad \sum_{e_{i_k} \in \; \ee{\oga_{i_1, \dots, i_{k-1}}}}
\frac{R_{i_k}}{L_{i_k}+R_{i_k}}\big(x(\oga_{i_1,\dots, i_k})-y(\oga_{i_1,\dots, i_k})\big).
\end{split}
\end{equation*}
\end{theorem}
\begin{proof}
We can use \lemref{lem succesive contraction of x-y} and \eqnref{eqn tau x and y} to obtain the result.
\end{proof}
Here is another formula for $r(\ga)$:
\begin{proposition}\label{prop succes contr dif}
Let $\ga$ be a bridgeless graph with $ 3 \leq v=\#(\vv{\ga)}$. Then
for any $k$ with $k+2 \leq v$,
\begin{equation}\label{eqn successive contraction of r}
\begin{split}
\frac{k(v-2)!}{(v-k-1)!} r(\ga)= \sum_{e_{i_1} \in
\ee{\ga}}\frac{R_{i_1}}{L_{i_1}+R_{i_1}} \; \; \dots
\sum_{\substack{e_{i_{k}} \in
\\ \ee{\oga_{i_1, \dots, i_{k-1}}}}}\frac{R_{i_{k}}}{L_{i_{k}}+R_{i_{k}}}
\sum_{t=1}^{k}L_{i_{t}}.
\end{split}
\end{equation}
\end{proposition}
\begin{proof}
By applying the second part of \corref{cor contraction1} successively, we obtain
\begin{equation}\label{eqn contraction r2}
\begin{split}
\frac{(v-2)!}{(v-k-2)!} r(\ga) &=
\sum_{e_{i_1} \in
\ee{\ga}}\frac{R_{i_1}}{L_{i_1}+R_{i_1}} \; \; \dots
\sum_{\substack{e_{i_{k}} \in
\\ \ee{\oga_{i_1, \dots, i_{k-1}}}}}\frac{R_{i_{k}}}{L_{i_{k}}+R_{i_{k}}}
r(\oga_{i_1, \dots, i_{k}}).
\end{split}
\end{equation}
Now, we can use induction on $k$ to show the identity in the proposition. When $k=1$, the result holds trivially by the definition of $r(\ga)$. Suppose the result is true for $k=n$ where $n+3 \leq v$. Let $A$ be the right hand side of \eqnref{eqn successive contraction of r} for $k=n+1$. By splitting the sum $\sum_{t=1}^{n+1}L_{i_{t}}=\big( \sum_{t=1}^{n}L_{i_{t}} \big)+L_{n+1}$ we have
\begin{equation*}\label{eqn contraction r3}
\begin{split}
A &=\sum_{\substack{e_{i_1} \in \\
\ee{\ga}}}\frac{R_{i_1}}{L_{i_1}+R_{i_1}} \; \; \dots
\sum_{\substack{e_{i_{n}} \in \\ \ee{\oga_{i_1, \dots, i_{n-1}}}}}
\frac{R_{i_{n}}}{L_{i_{n}}+R_{i_{n}}}\sum_{t=1}^{n}L_{i_{t}}
\sum_{\substack{e_{i_{n+1}} \in
\\ \ee{\oga_{i_1, \dots, i_{n}}}}}\frac{R_{i_{n+1}}}{L_{i_{n+1}}+R_{i_{n+1}}}
\\ & \qquad +
\sum_{\substack{e_{i_1} \in \\
\ee{\ga}}}\frac{R_{i_1}}{L_{i_1}+R_{i_1}} \; \; \dots
\sum_{\substack{e_{i_{n}} \in \\ \ee{\oga_{i_1, \dots, i_{n-1}}}}}
\frac{R_{i_{n}}}{L_{i_{n}}+R_{i_{n}}}r(\oga_{i_1, \dots, i_{n}})
\\ &=(v-n-1) \sum_{\substack{e_{i_1} \in \\
\ee{\ga}}}\frac{R_{i_1}}{L_{i_1}+R_{i_1}} \; \; \dots
\sum_{\substack{e_{i_{n}} \in \\ \ee{\oga_{i_1, \dots, i_{n-1}}}}}
\frac{R_{i_{n}}}{L_{i_{n}}+R_{i_{n}}}\sum_{t=1}^{n}L_{i_{t}}
+ \frac{(v-2)!}{(v-n-2)!} r(\ga)
\\ & \text{by \eqnref{eqn genus} applied to $\oga_{i_1, \dots, i_{n}}$, and by \eqnref{eqn contraction r2}.}
\\ &=\frac{(n+1)(v-2)!}{(v-n-2)!} r(\ga), \quad \text{by the induction assumption}.
\end{split}
\end{equation*}
Hence the result follows.
\end{proof}
Note that \eqnref{eqn contraction r2} generalizes the second equation in \corref{cor contraction1}.
\begin{corollary}\label{cor succes contr dif}
Let $\ga$ be a bridgeless graph with $ 3 \leq v=\#(\vv{\ga)}$. Then
\begin{equation*}
\begin{split}
(v-2) (v-2)! r(\ga) =&\sum_{\substack{e_{i_1} \in \\
\ee{\ga}}}\frac{R_{i_1}}{L_{i_1}+R_{i_1}} \; \; \dots
\sum_{\substack{e_{i_{v-2}} \in
\\ \ee{\oga_{i_1, \dots, i_{v-3}}}}}\frac{R_{i_{v-2}}}{L_{i_{v-2}}+R_{i_{v-2}}}
\sum_{t=1}^{v-2}L_{i_{t}}.
\end{split}
\end{equation*}
\end{corollary}
\begin{proof}
The result follows from \propref{prop succes contr dif} with
$k=v-2$.
\end{proof}
\begin{corollary}\label{cor succes contr difb}
Let $\ga$ be a bridgeless graph with $ 3 \leq v=\#(\vv{\ga)}$ and
$e$ edges. For any $k \in \{1,2, \dots, v-2\}$, let $A_k=\{
\sum_{t=1}^{k}L_{i_{t}}| \{ i_1, \dots, i_{k}\} \subseteq \{1,2,
\dots, e\} \}$. Let $C_{k}=\max{(A_k)} $ and $c_{k}=\min{(A_k)}
$. Then we have
$$ \frac{(v-1)}{k}c_k \leq r(\ga)
\leq \frac{(v-1)}{k}C_k,
\text{ and in particular, \quad}
\frac{v-1}{v-2}c_{v-2} \leq r(\ga)
\leq \frac{v-1}{v-2}C_{v-2}.$$
\end{corollary}
\begin{proof}
The result follows from \propref{prop succes contr dif} and
\eqnref{eqn genus}.
\end{proof}
Note also that successive application of the first part of \corref{cor contraction1} gives
\begin{equation}\label{eqn successive contraction z}
\begin{split}
\frac{(v-1)!}{(v-k-1)!} z(\ga) &=\sum_{\substack{e_{i_1} \in \\
\ee{\ga}}}\frac{R_{i_1}}{L_{i_1}+R_{i_1}} \; \; \dots
\sum_{\substack{e_{i_{k}} \in
\\ \ee{\oga_{i_1, \dots, i_{k-1}}}}}\frac{R_{i_{k}}}{L_{i_{k}}+R_{i_{k}}}
z(\oga_{i_1, \dots, i_{k}}).
\end{split}
\end{equation}
The following theorem generalizes \thmref{thm tau contract}.
\begin{theorem}\label{thm succes contr1}
Let $\ga$ be a bridgeless metrized graph with $(k+2) \leq v=\#(\vv{\ga)}$ for
some integer $k \geq 1$. Then
\begin{equation*}
\begin{split}
\ta{\ga} & =\frac{(v-k-2)!}{(v-2)!} \sum_{e_{i_1} \in \;
\ee{\ga}}\frac{R_{i_1}}{L_{i_1}+R_{i_1}}
\; \dots \sum_{\substack{e_{i_{k}} \in \\ \ee{\oga_{i_1, \dots, i_{k-1}}}}}
\frac{R_{i_k}}{L_{i_k}+R_{i_k}}\ta{\oga_{i_1,\dots, i_k}}
-\frac{k \cdot z(\ga)}{12(v-k-1)}.
\end{split}
\end{equation*}
\end{theorem}
\begin{proof}
First, we note that $\ta{\oga_{i_1,\dots, i_k}}=\frac{\ell(\oga_{i_1,\dots, i_k})}{12}-\frac{x(\oga_{i_1,\dots, i_k})-y(\oga_{i_1,\dots, i_k})}{6}$ by \eqnref{eqn tau x and y}, and $\ell(\oga_{i_1,\dots, i_k})=\ell(\ga)-\sum_{t=1}^{k}L_{i_{t}}$. Then the result follows by applying \thmref{thm succesive contraction of x-y}
and \propref{prop succes contr dif}.
\end{proof}
\begin{corollary}\label{cor succes contr1}
Let $\ga$ be a bridgeless metrized graph with $ 3 \leq v=\#(\vv{\ga)}$. Then
\begin{equation*}
\begin{split}
\ta{\ga} =\frac{1}{(v-2)!} \sum_{\substack{e_{i_1} \in \\ \ee{\ga}}}\frac{R_{i_1}}{L_{i_1}+R_{i_1}}
\dots \sum_{\substack{e_{i_{v-2}} \in \\ \ee{\oga_{i_1, \dots, i_{v-3}}}}}
\frac{R_{i_{v-2}}}{L_{i_{v-2}}+R_{i_{v-2}}}\ta{\oga_{i_1,\dots, i_{v-2}}}
-\frac{v-2}{12} z(\ga).
\end{split}
\end{equation*}
\end{corollary}
\begin{proof}
The result follows from \thmref{thm succes contr1}
with $k=v-2$.
\end{proof}
Recall that \thmref{thm tau contract} is valid for graphs with more than $2$ vertices.
If an edge $e_{i_k}$ is not a self loop in $\oga_{i_1,i_2, \dots, i_{k-1}}$, then $\#(\vv{\oga_{i_1,i_2, \dots, i_k}})= \#(\vv{\oga_{i_1,i_2, \dots, i_{k-1}}})-1$. We call $\oga_{i_1,i_2, \dots, i_{v-2}}$ be an \textit{admissible contraction} of $\ga$, if it is obtained from $\ga$ by contracting edges with distinct end points, i.e.,
if we have $\#(\vv{\oga_{i_1,i_2, \dots, i_{v-2}}})=2$. Note that such graphs are the only ones that contribute to the sum in \corref{cor succes contr1}.

Let $\oga_{i_1,i_2, \dots, i_{v-2}}$ be an admissible contraction of $\ga$, and let $\vv{\oga_{i_1,\dots, i_{v-2}}}=\{p, q\}$.
The graph $\oga_{i_1,\dots, i_{v-2}}$ has $n(i_1, \dots, i_{v-2})$ multiple edges between the vertices $p$ and $q$, and self-loops at $p$ or $q$. Figure \ref{fig contract1} illustrates $\oga_{i_1,\dots, i_{v-2}}$.
\begin{figure}
\centerline{\epsffile{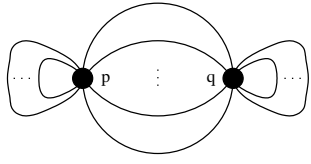}} \caption{A Banana graph with self loops.} \label{fig contract1}
\end{figure}
%
Let $n':=n(i_1, \dots, i_{v-2})$ be the number of multiple edges in $\oga_{i_1,\dots, i_{v-2}}$,
and let $B':=\{e_{j_1}, e_{j_2}, \dots e_{j_{n'}}\}$ be the set of multiple edges
in $\oga_{i_1,\dots, i_{v-2}}$. For the resistance function $r'(x,y)$ in $\oga_{i_1,\dots, i_{v-2}}$,
we have
$r'(p,q)=\frac{1}{\sum_{t=1}^{n'}\frac{1}{L_{j_t}}}$
by circuit theory.
Therefore,
\begin{proposition}\label{prop banana x and y}
Let $\ga$ be a bridgeless metrized graph. Using the notation above, for each admissible contraction
$\oga_{i_1, \dots, i_{v-2}}$ of $\ga$, we have
\begin{equation*}
\begin{split}
x(\oga_{i_1, \dots, i_{v-2}}) =(n'-1) \cdot r'(p,q), \qquad \quad y(\oga_{i_1, \dots, i_{v-2}}) = r'(p,q).
\end{split}
\end{equation*}
\end{proposition}
\begin{proof}
First note that $r'(p,q)=\frac{L_{t}R_{t}}{L_{t}+R_{t}}$ for each $e_t \in B'$, and $R_{i_{t}}=0$ if $e_{i_{t}} \not \in B'$.
\begin{equation}\label{eqn banana r}
\begin{split}
r(\oga_{i_1, \dots, i_{v-2}})= \sum_{e_{t} \; \in B'} \frac{L_{t}R_{t}}{L_{t}+R_{t}} = n' \cdot r'(p,q).
\end{split}
\end{equation}
Moreover, $(R_{a_{i},p}-R_{b_{i},p})^2=\ri^2$ for each $e_i \in \ee{\oga_{i_1, \dots, i_{v-2}}}$. Thus, by definition
$x(\oga_{i_1, \dots, i_{v-2}})=\sum_{e_{i_{v-1}} \in \ee{\oga_{i_1, \dots, i_{v-2}}}}
\frac{L_{i_{v-1}}^2R_{i_{v-1}}}{(L_{i_{v-1}}+R_{i_{v-1}})^2}=\sum_{e_{t} \; \in B'} \frac{L_{t}^2R_{t}}{(L_{t}+R_{t})^2}=
r'(p,q) \sum_{e_{t} \; \in B'} \frac{L_{t}}{L_{t}+R_{t}}=r'(p,q)(n'-1)$, where the last equality follows from \eqnref{eqn genus}.
This proves the first equality, and the second equality follows from the first equality and \eqnref{eqn banana r}.
\end{proof}
Here is another formula for the tau constant:
\begin{theorem}\label{thm succes main}
Let $\ga$ be a bridgeless metrized graph with $ 3 \leq v =\#(\ee{\ga})$. Let $p$, $q$, $n'$ and $B'$ be as defined above.
\begin{equation*}
\begin{split}
\ta{\ga} & =\frac{\ell(\ga)}{12}- \frac{1}{6 \cdot (v-2)!}
\sum_{e_{i_1} \in \; \ee{\ga}}\frac{R_{i_1}}{L_{i_1}+R_{i_1}}
\; \;
\dots \sum_{\substack{e_{i_{v-2}} \in \\ \ee{\oga_{i_1, \dots, i_{v-3}}}}}
\frac{R_{i_{v-2}}}{L_{i_{v-2}}+R_{i_{v-2}}} (n'-2) r'(p,q).
\end{split}
\end{equation*}
\end{theorem}
\begin{proof}
We have $x(\oga_{i_1, \dots, i_{v-2}})-y(\oga_{i_1, \dots, i_{v-2}}) =(n'-2) \cdot r'(p,q)$, by \propref{prop banana x and y}. Therefore, we obtain what we want by using \thmref{thm succesive contraction of x-y} with $k=v-2$.
\end{proof}

\section{Edge connectivity and the tau constant}\label{section edge connectivity}

In this section, we will prove that Conjecture
~\ref{TauBound} holds with $C=\frac{1}{108}$ for any graph $\ga$ with edge connectivity more than or equal to $6$, and
we will give a lower bound to the tau constant in terms of edge connectivity.

Let $\ga$ be a bridgeless metrized graph, and let $\oga_{i_1, \dots,
i_{v-2}}$, $n'$, $p$, $q$, $r'(p,q)$ and $B'$ be as in
\secref{section successive edge contraction}.
Recall that $n':=n(i_1, \dots, i_{v-2})$ is the
number of multiple edges in $\oga_{i_1,\dots, i_{v-2}}$ and that
$B':=\{e_{1}, e_{2}, \dots e_{n'}\}$ is the set of multiple edges in
$\oga_{i_1,\dots, i_{v-2}}$. We will show that a lower bound for
$$N(\ga):=\min \{ n' | \{i_1, \dots, i_{v-2}\}
\subset \{1,2, \dots, e \} \}.$$ gives a lower bound for $\tg$.
We will make some observations about $N(\ga)$ after recalling some basic definitions from graph theory.

We recall the following inequality between the edge connectivity $\Lambda(\ga)$
, vertex connectivity $\kappa(\ga)$, and the minimum degree of the valences $\delta(\ga)$.
\begin{remark}\label{edgecon-mindegree-vertexcon}
For a graph $\ga$, we have $\kappa(\ga) \leq \Lambda(\ga) \leq \delta(\ga)$
by basic graph theory \cite[pg. 3]{BB1}.
\end{remark}
Recall that a metrized graph is connected by definition.
\begin{lemma}\label{lem edge connectivity}
Let $\ga$ be a graph. Then $N(\ga)=\Lambda(\ga).$
\end{lemma}
\begin{proof}
If $\vv{\ga} = 2$, then $\ga$ is a banana graph with possibly self-loops.
Then $N(\ga)=\Lambda(\ga)$ clearly.

Note that when we contract an edge of a graph $\ga$ with $\vv{\ga} \geq 3$,
the edge connectivity either does not change or increases. Therefore, $\Lambda(\oga_{i_1, \dots,
i_{v-2}}) \geq \Lambda(\ga)$ for the contraction of any edges
$e_{i_1}, \dots, e_{i_{v-2}}.$ Since $n'=\Lambda(\oga_{i_1, \dots,
i_{v-2}}) \geq \Lambda(\ga)$, we have $N(\ga)\geq \Lambda(\ga)$.

Let $k=\Lambda(\ga)$, and let $e_1, e_2, \dots e_k$ be
edges such that $\ga-\{e_1, e_2, \dots e_k \}$ is disconnected but
$\ga-(\{e_1, e_2, \dots e_k \}-e_j)$ is connected for each $e_j$
where $1 \leq j \leq k$. Also, let $p$ and $q$ be the end points of
the edge $e_k$. Note that $e_k$ is a bridge in $\ga-\{e_1, e_2,
\dots e_{k-1} \}$. That is, $\ga-\{e_1, e_2, \dots e_{k-1} \}=\beta
\cup e_k \cup \gamma$ for some graphs $\beta$ and $\gamma$ with
$\beta \cap e_k=\{p \}$ and $\gamma \cap e_k=\{q \}$. Contract
edges in $\ee{\beta}$, say $e_{i_1}, e_{i_2}, \dots e_{i_s}$, until
$\beta$ has $1$ vertex. Similarly, contract edges in $\ee{\gamma}$,
say $e_{l_1}, e_{l_2}, \dots e_{l_t}$, until $\gamma$ has $1$
vertex. Then, $s+t=v-2$ and $n(i_1, i_2, \dots, i_{s}, l_1, l_2,
\dots, l_t)=k$ for the contraction graph $\oga_{i_1, i_2, \dots,
i_{s}, l_1, l_2, \dots, l_t}$. Thus, $N(\ga) \leq \Lambda(\ga).$

Hence, the result follows.
\end{proof}
We will need the following computation before we relate the edge connectivity $\Lambda(\ga)$
to $\ta{\ga}$.
\begin{corollary}\label{cor edge connectivity}
Let $\ga$ be a bridgeless metrized graph with genus $g$. Then for any admissible contraction $\oga_{i_1, \dots, i_{v-2}}$ of $\ga$ we have
$$g \cdot y(\oga_{i_1, \dots, i_{v-2}}) \geq x(\oga_{i_1, \dots, i_{v-2}}) \geq (\Lambda(\ga)-1) \cdot y(\oga_{i_1, \dots, i_{v-2}}).$$
\end{corollary}
\begin{proof}
Since $\oga_{i_1, \dots, i_{v-2}}$ has $e-(v-2)=g+1$ edges, $g+1 \geq \max \{ n' | \{i_1, \dots, i_{v-2}\}
\subset \{1,2, \dots, e \} \}$. Then the first inequality follows from \propref{prop banana x and y}.
The second inequality follows from \lemref{lem edge connectivity} and \propref{prop banana x and y}.
\end{proof}
When $k=v-2$, \eqnref{eqn successive contraction z} becomes
\begin{equation}\label{eqn z contr}
\begin{split}
(v-1)! z(\ga)= \sum_{\substack{e_{i_1} \in \\
\ee{\ga}}}\frac{R_{i_1}}{L_{i_1}+R_{i_1}}
\dots \sum_{\substack{e_{i_{v-2}} \in \\ \ee{\oga_{i_1, \dots,
i_{v-3}}}}} \frac{R_{i_{v-2}}}{L_{i_{v-2}}+R_{i_{v-2}}}
\sum_{\substack{e_{i_{v-1}} \in \\ \ee{\oga_{i_1, \dots, i_{v-2}}}}}
\frac{L_{i_{v-1}}^2}{L_{i_{v-1}}+R_{i_{v-1}}}.
\end{split}
\end{equation}
\begin{lemma}\label{lem z contr2}
For each admissible contraction $\oga_{i_1, \dots, i_{v-2}}$ of $\ga$ as above we have
\begin{equation*}
\begin{split}
\sum_{e_t \in B'}\frac{L_t^2}{L_t+R_t}\geq n' \cdot (n'-1) r'(p,q).
\end{split}
\end{equation*}
\end{lemma}
\begin{proof}
We have $\frac{1}{n'}\sum_{e_t \in B'}L_t \geq \frac{n'}{\sum_{e_t \in B'}\frac{1}{L_t}}$
by Arithmetic-Harmonic Mean inequality. On the other hand, $\sum_{e_t \in B'}L_t=\sum_{e_t \in B'}\frac{L_t^2}{L_t+R_t}+\sum_{e_t \in B'}\frac{L_t R_t}{L_t+R_t}$, and $r'(p,q)=\frac{1}{\sum_{t=1}^{n'}\frac{1}{L_{j_t}}}$. Thus the result follows from \eqnref{eqn banana r}.
\end{proof}
\begin{lemma}\label{lemcor z contr}
Let $\ga$ be a bridgeless metrized graph. Then we have
$z(\ga) \geq  \frac{\Lambda(\ga)}{v-1} x(\ga).$
\end{lemma}
\begin{proof}
If we apply the first part of \thmref{thm contraction} successively,
we derive the following expression:
\begin{equation}\label{eqn cont for x}
\begin{split}
x(\ga)=\frac{1}{(v-2)!} \sum_{e_{i_1} \in
\ee{\ga}}\frac{R_{i_1}}{L_{i_1}+R_{i_1}}
\dots \sum_{\substack{e_{i_{v-2}} \in \\ \ee{\oga_{i_1, \dots,
i_{v-3}}}}} \frac{R_{i_{v-2}}}{L_{i_{v-2}}+R_{i_{v-2}}}
x(\oga_{i_1, \dots, i_{v-2}}).
\end{split}
\end{equation}
On the other hand, for each admissible contraction $\oga_{i_1, \dots, i_{v-2}}$ of $\ga$,
$x(\oga_{i_1, \dots, i_{v-2}})=(n'-1) r(p,q)$
by \propref{prop banana x and y}. Then the result follows
from \eqnref{eqn z contr}, \lemref{lem edge connectivity}, \lemref{lem z contr2}, and \eqnref{eqn cont for x}.
\end{proof}
Set $$w(\ga):=\frac{1}{(v-2)!} \sum_{\substack{e_{i_1} \in \\
\ee{\ga}}}\frac{R_{i_1}}{L_{i_1}+R_{i_1}}
\dots \sum_{\substack{e_{i_{v-2}} \in \\ \ee{\oga_{i_1, \dots,
i_{v-3}}}}} \frac{R_{i_{v-2}}}{L_{i_{v-2}}+R_{i_{v-2}}}
\sum_{\substack{e_{i_{v-1}} \in \\ \ee{\oga_{i_1, \dots, i_{v-2}}}}}
\frac{L_{i_{v-1}}^3}{(L_{i_{v-1}}+R_{i_{v-1}})^2}. $$
Then we have
\begin{lemma}\label{lem z contr22}
Let $\ga$ be a bridgeless metrized graph.
Then $(v-1)z(\ga) = w(\ga)+x(\ga)$.
\end{lemma}
\begin{proof}
The result follows from Equations (\ref{eqn z contr}) and (\ref{eqn cont for x}).
\end{proof}
\begin{theorem}\cite[Theorem 2.26]{C2}\label{thm2term}
Let $\ga$ be a normalized metrized graph. Then
$$\sum_{e_i \in \, \ee{\ga}}\frac{\li\ri^2}{(\li+\ri)^2} \geq \Big(\sum_{e_i \in \, \ee{\ga}}\frac{\li\ri}{\li+\ri}\Big)^2.$$
\end{theorem}
%
\begin{lemma}\cite[Lemma 2.12]{C2}\label{lem2term}
Let $\ga$ be a metrized graph and $p \in \vv{\ga}$. Then if $e_i \sim p$ indicates that edge $e_i$
is incident to vertex $p$
\begin{equation*}
\begin{split}
\sum_{e_i \in \, \ee{\ga}}\frac{\li(R_{a_{i},p}-R_{b_{i},p})^2}{(\li+\ri)^2}
= \frac{2}{v}\sum_{e_i \in \, \ee{\ga}}\frac{\li\ri^2}{(\li+\ri)^2}
+ \frac{1}{v}\sum_{p \in \, \vv{\ga}}\Bigg(\sum_{\substack{e_i \not\sim p\\ e_i \in \,
\ee{\ga}}}\frac{\li(R_{a_{i},p}-R_{b_{i},p})^2}{(\li+\ri)^2} \Bigg).
\end{split}
\end{equation*}
\end{lemma}
We have the following relations between $x(\ga)$ and $y(\ga)$:
\begin{theorem}\label{thm edgecon11}
Let $\ga$ be a normalized bridgeless metrized graph with $\# (\vv{\ga})=v$, and let $x=x(\ga)$, $y=y(\ga)$.
Then we have
\begin{enumerate}
\item $\tg=\frac{1}{12}-\frac{x}{6}+\frac{y}{6},$
\item $1 \geq \frac{\Lambda(\ga)+v-1}{v-1} x+y,$ \, $x \geq 0$, \text{ and } $y \geq 0,$
\item $y \geq \frac{v+6}{4 v} (x+y)^2,$
\item $g \cdot y \geq x \geq (\Lambda(\ga)-1)y$.
\end{enumerate}
\end{theorem}
\begin{proof}
Since $\ga$ is normalized, $\ell(\ga)=1$. Thus, part $(1)$ follows from \eqnref{eqn tau x and y}.

Part $(2)$ follows from \lemref{lemcor z contr} and \eqnref{eqn sumof x and y}.

By \lemref{lem2term} and the definition of $y$, we have
\begin{equation}\label{eqn y lower bound}
\begin{split}
y \geq \frac{v+6}{4 v}\sum_{e_i \in \, \ee{\ga}}\frac{\li\ri^2}{(\li+\ri)^2}.
\end{split}
\end{equation}
Thus, part $(3)$ follows from \eqnref{eqn y lower bound} and \thmref{thm2term}.

We have  $g \cdot y(\oga_{i_1, \dots, i_{v-2}}) \geq x(\oga_{i_1, \dots, i_{v-2}}) \geq (\Lambda(\ga)-1) \cdot y(\oga_{i_1, \dots, i_{v-2}})$ by \corref{cor edge connectivity}. We inductively apply \thmref{thm contraction} to obtain
$$(v-2)! y(\ga) = \sum_{e_{i_1} \in \;
\ee{\ga}}\frac{R_{i_1}}{L_{i_1}+R_{i_1}} \sum_{\substack{e_{i_2} \in
\\ \ee{\oga_{i_1}}}}\frac{R_{i_2}}{L_{i_2}+R_{i_2}}
\; \dots
\sum_{\substack{e_{i_{v-2}} \in \\ \ee{\oga_{i_1, \dots, i_{v-3}}}}}
\frac{R_{i_{v-2}}}{L_{i_{v-2}}+R_{i_{v-2}}}y(\oga_{i_1,\dots, i_{v-2}}).$$
Thus, using \eqnref{eqn cont for x} we have part $(4)$.
\end{proof}
\begin{figure}
\centering
\includegraphics[scale=0.65]{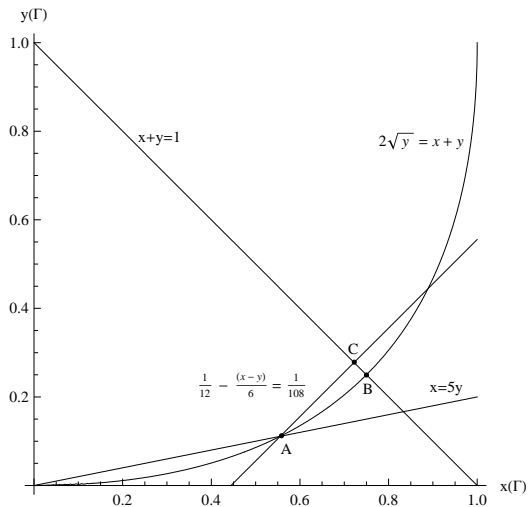} \caption{The lower bound for $\tg$ when $\Lambda(\ga)\geq 6$ and $v\rightarrow \infty$. $A=(\frac{5}{9},\frac{1}{9})$, $B=(\frac{3}{4},\frac{1}{4})$, $C=(\frac{13}{18},\frac{5}{18})$} \label{fig ParabolicLB2}
\end{figure}
Now, we can state the main result of this paper:
\begin{theorem}\label{thm edgecon2}
Let $\ga$ be a metrized graph with $v$ vertices. Then we have
\begin{enumerate}
\item $\tg \geq \elg \big(\frac{1}{12}(1-\frac{4}{\Lambda(\ga)})^2 +\frac{4(\Lambda(\ga)-2)}{(v+6)\Lambda(\ga)^2}\big), \quad \text{if $\Lambda(\ga) \geq 4$}$.
In particular, $\tg \geq \frac{\elg}{108}$ if $\Lambda(\ga) \geq 6$, and $\tg \geq \frac{\elg}{300}$ if $\Lambda(\ga) = 5$.
\item $ \tg \geq \frac{\elg}{2(v+6)}.$
In particular, $\tg \geq \frac{\elg}{108}$ if $v \leq 48$.
\end{enumerate}
\end{theorem}
\begin{proof}
If an edge $e_i \in \ee{\ga}$ is a bridge of length $\li$, then it contributes to $\tg$ by $\frac{\li}{4}$ (see \cite[Corollaries 2.22 and 2.23]{C2} for more information). Therefore, we can assume that $\ga$ is bridgeless. On the other hand, by using the scale-indepence of the tau constant (see \remref{rem tau scale-idependence}), we can assume that $\ga$ is normalized.

Now, we look for $x$ and $y$ values that satisfy the inequalities in parts $(2)$, $(3)$, and $(4)$ of
\thmref{thm edgecon11} and minimize $\frac{1}{12}-\frac{x}{6}+\frac{y}{6}$.

Whenever $\Lambda(\ga) \geq 4$, by elementary calculus, we see that
the line $x=(\Lambda(\ga)-1)y$ and the parabola $y=\frac{v+6}{4 v} (x+y)^2$
intersect at the point with coordinates $x =
\frac{4 v (\Lambda(\ga)-1)}{(v+6)\Lambda(\ga)^2}$ and $y =
\frac{4 v}{(v+6) \Lambda(\ga)^2}$, and that these give a lower bound to $\frac{1}{12}-\frac{x}{6}+\frac{y}{6}$.
This proves the first inequality in part $(1)$.

Again by elementary calculus, we see that the line $\frac{1}{12}-\frac{x}{6}+\frac{y}{6}=c$ is tangential to the parabola $y=\frac{v+6}{4 v} (x+y)^2$ at the point with coordinates $x=\frac{3 v}{4 (v+6)}$ and $y=\frac{v}{4 (v+6)}$, and that these give a lower bound to $\frac{1}{12}-\frac{x}{6}+\frac{y}{6}$.
This proves the first inequality in part $(2)$.

The remaining parts are immediate from what we have shown.
\end{proof}
\begin{theorem}\label{thmcor tau equal length}
Let $\ga$ be a normalized bridgeless metrized graph. If all the edge lengths are equal to each other, then we have
$$\frac{1}{12}-\frac{v-1}{6e}+ \frac{v-1}{3 e \Lambda(\ga)} \geq \tg \geq
\frac{1}{12}-\frac{v-1}{6e}+\frac{v+6}{12v}\big(\frac{v-1}{e}\big)^2.$$
In particular, if $\ga$ is an $n$-regular metrized graph and $\Lambda(\ga)=n$, we have
$$ \frac{1}{12}-\frac{(v-1)(n-2)}{3vn^2}\geq \tg \geq
\frac{1}{12}\big( \frac{g}{e}\big)^2+\frac{1}{2 v}\big(\frac{v-1}{e}\big)^2
.$$
\end{theorem}
\begin{proof}
Since $\li=\frac{1}{e}$ for each edge, $x(\ga)+y(\ga)=\frac{v-1}{e}$ by \eqnref{eqn genus} and \eqnref{eqn sumof x and y}.
Therefore, parts $(3)$ and $(4)$ of \thmref{thm edgecon11} are equivalent to $\frac{v-1}{\Lambda(\ga) e} \geq y \geq \frac{v+6}{4 v}(\frac{v-1}{e})^2$, and \eqnref{eqn tau x and y} is equivalent to $\tg=\frac{1}{12}-\frac{v-1}{6 e}+\frac{y}{3}$. These give the first two inequalities. The final two inequalities follow from the fact that $e=\frac{n v}{2}$ when $\ga$ is $n$-regular.
\end{proof}


\section{Cubic graphs}\label{section cubic graphs}

In this section, we will show that Conjecture ~\ref{TauBound} holds
for all metrized graphs if it holds for cubic metrized graphs. We call a
$3$-regular metrized graph a ``cubic metrized graph'' or ``cubic graph'' for short.
We will consider the metrized graphs with $\kappa(\ga) \geq 2$
where $\kappa(\ga)$ is the vertex connectivity. By
\remref{remcutvertex1}, this would be enough to prove Conjecture
~\ref{TauBound}.

We will use the following notation and graph
constructions.

Suppose $\ga$ is a normalized metrized graph, i.e., $\ell(\ga)=1$, and $p \in
\vv{\ga}$ is a vertex with valence $n\geq 4$. We want to transform
$\ga$ into another normalized metrized graph, $\ga_{p,(n-3)}^N$, by adding
new edges and new vertices of valence $3$ to $\ga$ in such a way
that the valence of the vertex $p$ becomes $3$ in $\ga_{p,(n-3)}^N$.
In $\ga_{p,(n-3)}^N$, we add $n-3$ new vertices $p^1, p^2, \dots,
p^{n-3}$ and $n-3$ new edges $e_{p,1}, e_{p,2}, \dots, e_{p,(n-3)}$
with pairs of end points $\{p^1,p^2\}, \{p^2,p^3\}, \dots, \{p^{n-3},p\}$, respectively.
Figure \ref{fig totalres2a} shows the details of the transformation.
\begin{figure}
\centerline{\epsffile{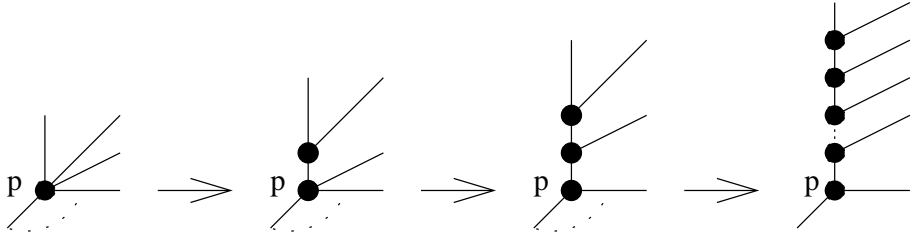}} \caption{Transforming a non-cubic graph $\ga$ to a cubic graph.} \label{fig totalres2a}
\end{figure}
The first graph in Figure \ref{fig totalres2a} shows $\ga$.

Suppose the edges with end point $p$ are given in a specified order. We disconnect the
first and the second edges from $p$. Then we reconnect them to $p$
via adding edge $e_{p,1}$, with end points $\{p^1,p\}$ and of length
$\varepsilon_{p,1}$, so that the new vertex $p^1$ becomes the end point
of the first edge, the second edge and the new edge $e_{p,1}$. We denote
this graph by $\ga_{p,1}$. Note that
$\ell(\ga_{p,1})=\ell(\ga)+\varepsilon_{p,1}=1+\varepsilon_{p,1}$
and if we contract the new edge $e_{p,1}$, we obtain $\ga$. Also,
the valence of $p$ in $\ga_{p,1}$ is $n-1$. Then we obtain
$\ga_{p,1}^N$ by normalizing $\ga_{p,1}$. $\ga_{p,1}^N$ is the
second graph in Figure \ref{fig totalres2a}. Note that the graphs $\ga_{p,1}$
and $\ga_{p,1}^N$ have the same shape, i.e. the same topology. At
the next step, we disconnect $e_{p,1}$ and the third edge with
vertex $p$ from $p$, then we reconnect them via adding the edge
$e_{p,2}$, with end points $\{p^2,p\}$ and of length
$\varepsilon_{p,2}$, so that the new vertex $p^2$ becomes the end point
of third edge, $e_{p,1}$ and $e_{p,2}$. We denote this graph by
$\ga_{p,2}$. Note that the valence of $p$ in $\ga_{p,2}$ is $n-2$.
Then by normalizing $\ga_{p,2}$ we obtain $\ga_{p,2}^N$ which is
shown by the third graph in Figure \ref{fig totalres2a}. We continue this process
until the valence of $p$ becomes $3$, i.e., until we obtain the
graphs $\ga_{p,(n-3)}$ and $\ga_{p,(n-3)}^N$.

Note that $\varepsilon_{p,k} >0$ for each $k=1, 2, \dots, n-3$.
Since $\kappa(\ga) \geq 2$, $\ga_{p,k}-{e_{p,k}}$ is connected for
each $k=1, 2, \dots, n-3$. Let $\ga_{p,0}^N:=\ga$.
\begin{lemma}\label{lemconcubic}
Let $k \in \{0,1,\dots, n-4\}$ and let $\ga_{p,k+1}^N$,
$\ga_{p,k}^N$, $p$ and $\varepsilon_{p,k+1}$ be as above. Then
\begin{equation*}
\ta{\ga_{p,k+1}^N} \leq \ta{\ga_{p,k}^N} +
\frac{\varepsilon_{p,k+1}}{1+\varepsilon_{p,k+1}} \cdot
(\frac{1}{12}-\ta{\ga_{p,k}^N}).
\end{equation*}
\end{lemma}
\begin{proof}
Let $e_{p,k+1}$, $p^k$, $p^{k+1}$, $\ga$, $\ga_{p,k}^N$,
$\ga_{p,k+1}$, $\ga_{p,k+1}^N$, $\varepsilon_{p,k+1}$ be as above.

Note that we can obtain $\ga_{p,k}^N$ from $\ga_{p,k+1}$ by
contracting the edge $e_{p,k+1}$ to its end points. Since
$\ga_{p,k}-{e_{p,k}}$ is connected, we can apply
\lemref{lemcontract1}. This gives
\begin{equation}\label{eqn concubic1}
\begin{split}
\ta{\ga_{p,k+1}}=\ta{\ga_{p,k}^N} + \frac{\varepsilon_{p,k+1}}{12}
-\frac{ \varepsilon_{p,k+1} A_k
}{\bar{R}_{k+1}(\varepsilon_{p,k+1}+\bar{R}_{k+1})},
\end{split}
\end{equation}
where $A_k:=A_{p^k,p^{k+1},\ga_{p, k+1}-{e_{p,k+1}}}$ and
$\bar{R}_{k+1}$ is the resistance, in $\ga_{p, k+1}-{e_{p,k+1}}$,
between $p^k$ and $p^{k+1}$.

Since $\ell(\ga_{p,k+1})=1+\varepsilon_{p,k+1}$,
\begin{equation}\label{eqn concubic2}
\begin{split}
\ta{\ga_{p,k+1}} =(1+\varepsilon_{p,k+1}) \cdot \ta{\ga_{p,k+1}^N}.
\end{split}
\end{equation}

Substituting \eqnref{eqn concubic2} into  \eqnref{eqn concubic1}
gives
\begin{equation}\label{eqn concubic3}
\begin{split}
\ta{\ga_{p,k+1}^N} &=\frac{\ta{\ga_{p,k}^N}}{1+\varepsilon_{p,k+1}}
+ \frac{\varepsilon_{p,k+1}}{1+\varepsilon_{p,k+1}} \cdot
(\frac{1}{12}-\frac{A_k}{\bar{R}_{k+1}(\varepsilon_{p,k+1}+\bar{R}_{k+1})})
\\ & = \ta{\ga_{p,k}^N}
+ \frac{\varepsilon_{p,k+1}}{1+\varepsilon_{p,k+1}} \cdot
(\frac{1}{12}-\frac{A_k}{\bar{R}_{k+1}(\varepsilon_{p,k+1}+\bar{R}_{k+1})}-
\ta{\ga_{p,k}^N})
\\ & \leq \ta{\ga_{p,k}^N}
+ \frac{\varepsilon_{p,k+1}}{1+\varepsilon_{p,k+1}} \cdot
(\frac{1}{12}- \ta{\ga_{p,k}^N}),
\end{split}
\end{equation}
since $A_k \geq 0$, \text{$\bar{R}_{k+1} > 0$ and
$\varepsilon_{p,k+1} > 0$}. This proves the result.
\end{proof}

\begin{theorem}\label{thmconjforcubics}
If there exists a positive constant $C$ such that
$ \ta{\beta}\geq C $ for any normalized cubic graph $\beta$, then
$ \ta{\ga} \geq C$ for any normalized graph $\ga$.
\end{theorem}
\begin{proof}
Let $\ga$ be an arbitrary normalized metrized graph. By the additive property of the tau constant
(\remref{remcutvertex1}) we can assume that $\ga$ has no cut vertices. If $\ga$ is a loop,
then $\tg=\frac{1}{12}$. Thus, we can assume that $\ga$ has a vertices with valence at least $3$.
After removing all vertices of valence $2$ from $\vv{\ga}$, we can assume that all vertices have
valence at least $3$. Suppose $\ga$ is not a cubic graph. Then by basic graph theory
%
%
$e> \frac{3}{2}v$, where $e=\#(\ee{\ga})$ and $v=\#(\vv{\ga})$. Let
$\varepsilon_0 :=\frac{\varepsilon}{2e-3v}$, for some arbitrary $\varepsilon > 0$.

Since $\ga$ is not cubic, there exists a vertex $p \in \vv{\ga}$
with $\va(p) \geq 4$. We construct the graphs $\ga_{p,k+1}$ and
$\ga_{p,k+1}^N$ for each $k=0,1, \dots \va(p)-4$ as mentioned at the beginning of this section.
In these constructions,  for each $k$ we take
\[
\varepsilon_{p,k+1}= \begin{cases}
\frac{\varepsilon_0}{\frac{1}{12}-\ta{\ga_{p,k}^N}}, & \text{if
$\frac{1}{12} \not =\ta{\ga_{p,k}^N} $ } \\
\text{a positive number}, \quad & \text{otherwise}.
\end{cases}
\]
Note that $\frac{1}{12} \geq \ta{\ga_{p,k}^N}$ by
\cite[Corollary 5.8]{C2}.
Then in both cases we have
\begin{equation}\label{eqn concubic4}
\begin{split}
\ta{\ga_{p,k+1}^N} \leq \ta{\ga_{p,k}^N} + \varepsilon_0 .
\end{split}
\end{equation}
By considering \eqnref{eqn concubic4} for each $k=0,1, \dots
\va(p)-4$, we obtain
\begin{equation*}\label{eqn concubic5}
\begin{split}
\ta{\ga_{p,\va(p)-3}^N} \leq \ta{\ga} + (\va(p)-3) \cdot
\varepsilon_0.
\end{split}
\end{equation*}

By following the same procedure for each $p \in \vv{\ga}$ with
$\va(p) \geq 4$, we obtain a normalized cubic graph $\beta$ such
that
\begin{equation*}\label{eqn concubic6}
\begin{split}
\ta{\beta} & \leq \ta{\ga} + \sum_{p \in \vv{\ga}}(\va(p)-3) \cdot
\varepsilon_0
= \ta{\ga} + (2e -3v) \cdot \varepsilon_0
= \ta{\ga} + \varepsilon.
\end{split}
\end{equation*}
Thus $\tg \geq C-\varepsilon$. Since $\varepsilon >0$ is arbitrary, $\tg \geq C$.
\end{proof}
\begin{remark}\label{remcubic}
\thmref{thmconjforcubics} shows that to prove Conjecture ~\ref{TauBound}, it is enough to establish it for cubic graphs.
\end{remark}
\begin{theorem}\label{thm edge connectivity 2 to 3}
Let $\ga$ be a metrized graph with $\Lambda(\ga)=2$. Then there exists a metrized graph $\beta$ such that $\tg=\ta{\beta}$, $\Lambda(\beta)\geq 3$, $\#(\ee{\ga}) \geq \#(\ee{\beta})$, and $g(\ga) = g(\beta)$.
\end{theorem}
\begin{proof}
Since $\Lambda(\ga)=2$, there is an edge $e_i \in \ee{\ga}$ such that $\Lambda(\ga-e_i)=1$, and let $\li$ be the length of $e_i$.
Let $C(e_i)=\{e_{i_1}, e_{i_2}, \cdots, e_{i_s} \}$ be the set of bridges in $\ga-e_i$, and let $L_{i_j}$ be the edge length of
$e_{i_j}$ for each $1 \leq j \leq s$. Let $\gamma$ be the metrized graph obtained from $\ga$ by contracting all of the edges in $C(e_i)$ to their end points, and by extending the length $\li$ of the edge $e_i$ to $\li+\sum_{j=1}^{s}L_{i_j}$.
We have $\elg=\ell(\gamma)$, and $\ta{\ga-e_i}=\ta{\gamma-e_i}+\frac{1}{4}\sum_{j=1}^{s}L_{i_j}$ by additive property of the tau constant (see \remref{remcutvertex1}), $R_i(\ga)=R_i(\gamma)+\sum_{j=1}^{s}L_{i_j}$ by elementary circuit reductions, and $L_i(\gamma)=\li+\sum_{j=1}^{s}L_{i_j}$ by our construction. Moreover, $A_{\pp,\qq,\ga-e_i}=A_{\pp,\qq,\gamma-e_i}$ by the additive property of $A_{p,q,\ga}$ (see \cite[Proposition 4.6]{C2}) and by \cite[Proposition 4.5]{C2}. By our construction,
$\#(\ee{\ga}) \geq \#(\ee{\gamma})$, and $g(\ga) = g(\gamma)$. If we apply \lemref{lemcor2twopunion} to
$\tg$ and $\ta{\gamma}$ and use the equalities we derived, we see that $\tg=\ta{\gamma}$.

Note that $\Lambda(\gamma-e_i) \geq 2$. If $\Lambda(\gamma)=2$, we apply the same process to $\gamma$. We can repeat this process
until we obtain a graph $\beta$ with the properties we wanted. \figref{fig edgeconnectivitytwo9} shows an example in which this process applied four times.
\end{proof}
\begin{figure}
\centering
\includegraphics[scale=1.4]{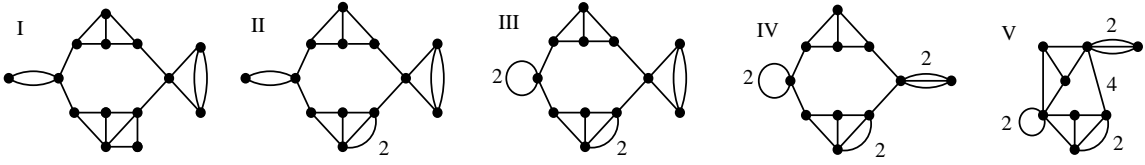}\caption{All of these graphs have equal tau constant. The last graph has edge connectivity $3$, and the others have edge connectivity $2$. The length of the extended edges are shown, the other edges have length $1$.} \label{fig edgeconnectivitytwo9}
\end{figure}
\begin{remark}\label{rem edge connectivity 2}
One of the implications of \thmref{thm edge connectivity 2 to 3} is that
if Conjecture ~\ref{TauBound} holds for metrized graphs with edge connectivity at least $3$, then it holds for
all metrized graphs.
\end{remark}

We show in \cite{C3} that $\tg$ can be computed by using the discrete Laplacian
of $\ga$ and its pseudo inverse.
In \cite{C6}, we construct families of metrized graphs with the tau constants between $\frac{\elg}{107}$ and $\frac{\elg}{108}$, and the computations suggest that we can have sequences of metrized graphs with the tau constants approaching (but not equal) to $\frac{\elg}{108}$.

Based on our theoretical and computational investigations, we refine Conjecture ~\ref{TauBound} as follows:
\begin{conjecture}\label{conj our}
For all metrized graphs $\Gamma$,
$\tau(\Gamma) > \frac{\ell(\Gamma)}{108}$.
\end{conjecture}



\begin{thebibliography}{999}

\bibitem[BB1]{BB1} B. Bollab\'as, {\em Extremal Graph Theory}, Dover Publications INC.,
Mineola, New York, 2004.

\bibitem[BF]{BF} M. Baker and X. Faber, Metrized graphs, Laplacian operators, and
electrical networks, Quantum graphs and their applications, 15--33,
{\em Contemp. Math.}, 415, Amer. Math. Soc., Providence, RI, 2006.


\bibitem[BR]{BRh} M. Baker and R. Rumely, Harmonic analysis on metrized graphs, {\em Canadian
J. Math}: May 9, 2005.

\bibitem[C1]{C1} Z. Cinkir, {\em The Tau Constant of Metrized Graphs},
Thesis at University of Georgia, 2007.

\bibitem[C2]{C2} Z. Cinkir, The Tau Constant of A Metrized Graph And Its Behavior Under Graph Operations, preprint,
\\http://arxiv.org/abs/0901.0407

\bibitem[C3]{C3} Z. Cinkir, The Tau Constant And The Discrete Laplacian of A Metrized Graph, preprint,
\\ http://arxiv.org/abs/0902.3401

\bibitem[C4]{C4} Z. Cinkir, Generalized Foster Identities And The Discrete Laplacian, preprint.

\bibitem[C5]{C5} Z. Cinkir, Bogomolov Conjecture Over Function Fields And Zhang's Conjecture, preprint,
\\ http://arxiv.org/abs/0901.3945

\bibitem[C6]{C6} Z. Cinkir, Metrized Graphs With Small Tau Constants, in preperation.


\bibitem[CR]{CR} T. Chinburg and R. Rumely, The capacity pairing,
{\em J. reine angew. Math.}  434 (1993), 1--44.


\bibitem[DS]{DS} Peter G. Doyle and J. Laurie Snell, {\em Random Walks and Electrical Networks},
Carus Mathematical Monographs, Mathematical Association of America,
Washington D.C., 1984. Available at \\http://arxiv.org/abs/math/0001057

\bibitem[REU]{REU} Summer 2003 Research Experience for
Undergraduates (REU) on metrized graphs at the University of Georgia.


\bibitem[Zh1]{Zh1} S. Zhang, Admissible pairing on a curve,
{\em Invent. Math.} 112 (1993), 171--193.

\bibitem[Zh2]{Zh2} S. Zhang, Gross--Schoen cycles and dualising sheaves, preprint,
\\ http://www.math.columbia.edu/$\sim$szhang/papers/Preprints.htm


\end{thebibliography}
\end{document}